\documentclass[11pt,reqno]{amsart}
\usepackage[text={160mm, 235mm},centering]{geometry}
\usepackage{float}
\usepackage{appendix}
\usepackage{graphicx}
\usepackage{chngpage}
\usepackage{subcaption}
\usepackage{multirow}
\usepackage[colorlinks=true, citecolor=blue, urlcolor=blue, linkcolor=blue]{hyperref}
\usepackage[all,cmtip]{xy}
\usepackage{amssymb}
\usepackage{tikz-cd}
\usepackage{amsmath}
\usepackage{relsize}
\usepackage{longtable}
\usepackage{pstricks}
\usepackage{color}
\usepackage[all]{xy}
\theoremstyle{plain}
\newtheorem{thm}{Theorem}[section]
\newtheorem{theorem}[thm]{Theorem}

\newtheorem{lemma}[thm]{Lemma}
\newtheorem{corollary}[thm]{Corollary}
\newtheorem{proposition}[thm]{Proposition}
\theoremstyle{definition}
\newtheorem{remark}[thm]{Remark}

\newtheorem{definition}[thm]{Definition}

\newtheorem{example}[thm]{Example}

\newtheorem{question}[thm]{Question}

\newtheorem{setup}[thm]{Set-up}

\numberwithin{equation}{section}
\newcommand{\romea}{\uppercase\expandafter{\romannumeral1}}
\newcommand{\romeb}{\uppercase\expandafter{\romannumeral2}}

\newcommand{\Aut}{{\rm Aut}}

\newcommand{\Diag}{{\rm diag}}

\newcommand{\GL}{{\rm GL}}
\newcommand{\PGL}{{\rm PGL}}
\newcommand{\PSL}{{\rm PSL}}

\newcommand{\SL}{{\rm SL}}
\newcommand{\JC}{{\rm JC}}

\newcommand{\Oplus}{\mathop{\ensuremath{\vcenter{\hbox{\scalebox{1.2}{$\bigoplus$}}}}}}
\newcommand{\C}{{\mathbb C}}

\renewcommand{\P}{{\mathbb P}}

\makeatletter
\@namedef{subjclassname@2020}{%
	\textup{2020} Mathematics Subject Classification}
\makeatother

\begin{document}

	\title{On automorphism groups of smooth hypersurfaces}
	
	\author{Song Yang, Xun Yu \and Zigang Zhu}
	
	\address{Center for Applied Mathematics and KL-AAGDM, Tianjin University, Weijin Road 92, Tianjin 300072, P.R. China}
	
	\email{syangmath@tju.edu.cn, xunyu@tju.edu.cn, zhzg0313@tju.edu.cn}
	
	
	\maketitle
	
	\begin{abstract}
		We show that smooth hypersurfaces in complex projective spaces with automorphism groups of maximum size are isomorphic to Fermat hypersurfaces, with a few exceptions. For the exceptions, we give explicitly the defining equations and automorphism groups.
	\end{abstract}
	\setcounter{tocdepth}{1}
	\tableofcontents

	\section{Introduction}\label{Intro}
	The purpose of this paper is to study smooth hypersurfaces in complex projective spaces with large automorphism groups. Let $F\in\C[x_1,\ldots,x_{n+2}]$ be a smooth homogeneous polynomial of degree $d\ge 3$ and let $X_F\subset \P^{n+1}$ denote the smooth hypersurface defined by $F$, where $n\ge 1$. An automorphism $f\in \Aut(X_F)$ of $X_F$ is called {\it linear} with respect to the embedding $X_F\hookrightarrow \P^{n+1}$ if $f$ extends to an automorphism of $\P^{n+1}$, i.e., $f$ is given by a linear change of homogeneous coordinates. The automorphism group $\Aut(X_F)$ is finite and equal to its subgroup ${\rm Lin}(X_F)$ of linear automorphisms if $(n,d)\neq (1,3), (2,4)$ (see \cite[Theorems 1 and 2]{MM63} for $n\ge 2$ and \cite[Theorem 2]{Cha78} for $n=1$). We use $\Aut(F)$ to denote the group of linear transformations that  preserve $F$ and $\mathcal{I}_{n+2,d} $ to denote the subgroup of the general linear group $\GL(n+2,\C)$ consisting of $d$ scalar matrices. Then clearly $\mathcal{I}_{n+2,d} $ is a normal subgroup of $\Aut(F)$ and the following exact sequence gives a basic relation between ${\rm Lin}(X_F)$ and $\Aut(F)$: 
	\begin{equation}\label{eq:keyexact}
		1\rightarrow \mathcal{I}_{n+2,d}  \longrightarrow\Aut(F) \longrightarrow {\rm Lin}(X_F)\rightarrow 1.
	\end{equation}
	
	Bounding the size of $\Aut(X_F)$, ${\rm Lin}(X_F)$ or $\Aut(F)$  is almost the same thing and has a long history (see \cite[Section 6]{OS78} for some historical remarks).  Bott--Tate and Orlik--Solomon showed that there exists an upper bound for the order $|\Aut(F)|$ depending only on $n$ and $d$ (see \cite[Corollary (2.7)]{OS78}). Finding an effective bound in terms of $n$ and $d$ is important.
	We use $X_d^n$ to denote the Fermat hypersurface in $\P^{n+1}$ of degree $d$ defined by Fermat polynomial $F_d^n:=x_1^d+x_2^d+\cdots+x_{n+2}^d$. It is known that $\Aut(X_d^n)$ is isomorphic to  a semidirect product $C_d^{n+1}\rtimes S_{n+2}$ and $|\Aut(X_d^n)|=d^{n+1}\cdot (n+2)!$ if $(n,d)\neq (1,3), (2,4)$ (see \cite[Page 147]{HS80}, \cite{Shi88}). By a classical theorem of Hurwitz, the automorphism group of a complex curve of genus $g\ge 2$ is of order at most $84 (g-1)$. This implies that for a plane curve $X_F$ with $d\geq 4$, $|\Aut(X_F)|\leq 42(d-3)d$. Automorphism groups of plane curves have been extensively studied by many people. For $n=1$ and $d\ge 4$,  plane curves $X_F$ with maximum $|\Aut(X_F)|$ are isomorphic to Fermat curves $X_{d}^1$ (resp. Klein quartic, resp. Wiman sextic) if $d\neq 4, 6$ (resp. $d=4$, resp. $d=6$) (see \cite[Theorem 1]{Pam13},  \cite[Theorem 2.5]{Haru19} and references therein). Andreotti \cite{And50} was the first to establish explicit bounds for the order of the automorphism group of varieties of general type in higher dimensions. For $n\ge 2$, Howard--Sommese \cite[Theorem 2]{HS80} proved that there is a constant $k_n$ depending only on $n$ such that $|\Aut(X_F)|\leq k_n\cdot d^{n+1}$ if $(n,d)\neq (2,4)$. Xiao (\cite[Theorem 1]{Xia94}, \cite[Theorem 2]{Xia95}) proved that a minimal smooth projective surface $S$ of general type has at most $42^2 K_S^2$ automorphisms. Consequently, for a surface $X_F\subset \P^3$ with $d\geq5$, $|\Aut(X_F)|\leq 42^2(d-4)^2d$. For $2\le n\le 5$, Fermat cubic $n$-fold has the largest possible order for the automorphism group among all smooth cubic $n$-folds (see \cite[\S 100]{Seg42},  \cite[Theorem 5.3]{Hos97}, \cite[Theorem 9.5.8]{Do12} for $n=2$; \cite[Theorem 1.1]{WY20} for $n=3$; \cite[Corollary 6.14]{LZ22}, \cite[Theorem 1.2]{YYZ23} for $n=4$; \cite[Theorem 1.1]{YYZ23} for $n=5$). Fermat quintic $3$-fold has the largest possible order for the automorphism group among all smooth quintic $3$-folds (\cite[Theorem 2.2]{OY19}). For some partial results on abelian subgroups of automorphism groups of smooth hypersurfaces of arbitrary dimension, see  for instance \cite[Theorem 0.1]{CS95},  \cite[B\'{e}zout Lemma]{Sza96}, \cite[Theorem 2.6]{GL11}, \cite{GL13}, \cite[Theorem 4.8]{Zhe22}, \cite{GLM23}.
	However, to the best of our knowledge,  for $n\ge 2$ except $(n,d)=(2,3), (2,4), (3,3),(4,3),(5,3)$ and $(3,5)$, the maximum order $|\Aut(X_F)|$ is still unknown (see e.g. \cite[Chapter 1, Remark 3.20]{Huy23}, \cite[Remark 2.11]{GLMV24}, \cite[Question 3.9]{Ess24}) and  the classification of groups of linear automorphisms of hypersurfaces is in a rudimentary
	state as also mentioned in \cite[Page 6]{Dol19}. Our main result is to completely classify smooth hypersurfaces with automorphism groups of maximum order, which in particular gives an optimal upper bound for $|\Aut(X_F)|$.
	
	\begin{theorem}\label{thm:Main}
		Fix integers $n\geq1$, $d\geq3$ with $(n,d)\neq(1,3),(2,4)$. Let $X\subset\P^{n+1}$ be a smooth hypersurface of degree $d$ with maximum $|\Aut(X)|$. Then $$|\Aut(X)|=d^{n+1}\cdot (n+2)!\; \text{ and }\; X \text{ is isomorphic to Fermat hypersurface } X_d^n$$ with the following exceptions:
		\begin{footnotesize}
			\begin{table}[H]\rm
				\renewcommand\arraystretch{1.2}
				\begin{tabular}{p{0.27in}p{0.74in}p{0.52in}p{4.21in}}
					\hline
					$(n,d)$&$\Aut(X)$&$|\Aut(X)|$&$X \text{ is isomorphic to } X_F$\\
					\hline
					$(1,4)$&$\PSL(2,7)$&$168$&$F=x_1^3x_2+x_2^3x_3+x_3^3x_1$ \\
					%                \hline
					$(1,6)$&$A_6$&$360$&$F=10x_1^3x_2^3+9(x_1^5+x_2^5)x_3-45x_1^2x_2^2x_3^2-135x_1x_2x_3^4+27x_3^6$ \\
					%                \hline
					$(2,6)$&$C_{6}.(S_4^2\rtimes C_2)$&$6912$&$F=x_1^5x_2+x_2^5x_1+x_3^5x_4+x_4^5x_3$ \\
					%                \hline
					$(2,12)$&$C_{12}.(A_5^2\rtimes C_2) $&$86400$&$F=x_1^{11}x_2+11x_1^6x_2^6-x_1x_2^{11}+x_3^{11}x_4+11x_3^6x_4^6-x_3x_4^{11}$ \\
					%                \hline
					$(4,6)$&${\rm PSU}(4,3).C_2$&$6531840$&$F=\sum\limits_{1\leq i\leq6}x_i^6+\sum\limits_{1\leq i\neq j\leq6}15x_i^4x_j^2-\sum\limits_{1\leq i<j<k\leq6}30x_i^2x_j^2x_k^2+240\sqrt{-3}x_1x_2x_3x_4x_5x_6$ \\
					%                \hline
					$(4,12)$&$C_{12}^2.(A_5^3\rtimes S_3) $&$186624000$&$F=x_1^{11}x_2+11x_1^6x_2^6-x_1x_2^{11}+x_3^{11}x_4+11x_3^6x_4^6-x_3x_4^{11}+x_5^{11}x_6+11x_5^6x_6^6-x_5x_6^{11}$ \\
					\hline
				\end{tabular}
			\end{table}
		\end{footnotesize}
	\end{theorem}	
	
	The sextic $F$ and its automorphism group in the case $(n,d)=(4,6)$ were given by Todd \cite[Section 6]{Tod50}. It seems that $X_F$ in the cases $(n,d)=(2,6), (2,12), (4,12)$ in the table above are previously unknown. It is known that automorphism groups of smooth hypersurfaces of general type are equal to their birational automorphism groups.
	Hacon--${\rm M^cKernan}$--Xu \cite[Theorem 1.1]{HMX13} showed that the number of birational automorphisms of a variety of general type $X$ is bounded from above by $c \cdot {\rm vol}(X, K_X )$, where $c$ is a constant that only depends on the dimension of $X$, and they asked for finding an explicit bound for the constant $c$ (\cite[Question 1.2]{HMX13}). As a direct consequence of Theorem \ref{thm:Main}, we obtain the optimal values of such constants for smooth hypersurfaces of general type.
	
	\begin{corollary}
		Let $X\subset \P^{n+1}$ be a smooth hypersurface of general type, where $n\ge 2$. Then
		$$|\Aut(X)|\leq (n+2)!(n+3)^n\cdot {\rm vol}(X,K_X).$$
		Moreover, the equality holds if and only if $X$ is isomorphic to Fermat hypersurface $X_{n+3}^n$.
	\end{corollary}
	
	Next we briefly explain the ideas of the proof of Theorem \ref{thm:Main}. In fact, we completely classify smooth hypersurfaces $X_F$ with $|\Aut(X_F)|\ge |\Aut(X_d^n)|$ (Theorem \ref{thm:atleastFermat}), which clearly implies Theorem \ref{thm:Main}. In order to prove Theorem \ref{thm:atleastFermat}, by the exact sequence \eqref{eq:keyexact}, it suffices to classify $X_F$ with $|\Aut(F)|\ge |\Aut(F_d^n)|$. A famous theorem of Jordan \cite{Jor78} says that, for every integer $r\geq1$, there exists a constant depending only on $r$ such that any finite subgroup of the general linear group $\GL(r,\C)$ has an abelian normal subgroup of index at most this constant. Collins \cite[Theorems A, B and D]{Col07} gave the optimal values ${\rm J}(r)$ of such constants for all $r\ge 1$.  By Jordan's Theorem, as a finite linear group in $\GL(n+2,\C)$, $\Aut(F)$ has a normal abelian subgroup $N$ of index at most ${\rm J}(n+2)$, and $|\Aut(F)|$ is bounded from above by the product $|N|\cdot {\rm J}(n+2)$. Hence bounding the size of $N$ leads to an upper bound of $|\Aut(F)|$. The known bounds for $|\Aut(F)|$ or $|\Aut(X_F)|$ mentioned previously are obtained in this way or a similar way (see \cite[Corollary 2.7]{OS78}, \cite[\S 3]{HS80}). However, it is often that when the index $[\Aut(F): N]$ gets bigger, the size of $N$ gets smaller. Thus the product of ${\rm J}(n+2)$ and the largest possible value for $|N|$ rarely gives an optimal bound for $|\Aut(F)|$ (see Remark \ref{rem:comparebounds}). To overcome this difficulty, we introduce two new notions, {\it canonical bound} (Definition \ref{def:canonical}) and {\it Fermat-test ratio} (Definition \ref{def:FTR}). Roughly speaking, canonical bounds integrate bounding the size of abelian subgroups of $\Aut(F)$ with bounding the size of certain primitive groups called {\it primitive constituents} (see $\bf{Set\textrm{-}up}$ \ref{setup}) of $\Aut(F)$, which turns out to be quite effective in our study of bounding $|\Aut(F)|$.
	More precisely, based on complete reducibility of linear representations of finite groups and the classical notion of primitive (projective) linear groups (see Section \ref{ss:abcan}), we introduce two associated exact sequences \eqref{eq:1associ}-\eqref{eq:2associ} for any finite linear group. Combining the associated exact sequences  with our bound for the order of abelian subgroups of $\Aut(F)$ (Theorem \ref{thm:diag}), we show that $|\Aut(F)|$ is at most the canonical bound $B(\Aut(F))$ of $\Aut(F)$ (Lemma \ref{lem:canonical}).  We say an invertible matrix $A$ is {\it semi-permutation} if $A$ is a diagonal matrix up to permutation of columns. Using Theorem \ref{thm:diag}, we immediately get an optimal bound for $|\Aut(F)|$ in the case $\Aut(F)$ consisting of semi-permutations (Theorem \ref{thm:mono}). 
	
	To handle the case $\Aut(F)$ consisting of  not only semi-permutations, we use both canonical bounds and Fermat-test ratios. As a consequence of Jordan's theorem, for each integer $r\ge 1$, there are only finitely many finite primitive groups in $\PGL(r,\C)$ (up to conjugation), and Collins \cite[Theorem A]{Col08} determined the maximum order ${\rm JC}(r)$ among such groups using the well-known classification of finite simple groups. Based on the definition of Fermat-test ratios and the explicit values of ${\rm JC}(r)$, we show that Fermat-test ratios have many nice properties (see Section \ref{ss:FTR}) which are crucial for our classification. In particular, using Fermat-test ratios, we quickly prove that Theorem \ref{thm:atleastFermat} holds if $n\ge 26$ or $d\ge 18$ (Theorem \ref{thm:boundnd}). Sections \ref{ss:refinedbounds}, \ref{ss:prim} and \ref{ss:proof} are devoted to the proof of Theorem \ref{thm:atleastFermat} for the remaining (finitely many) pairs $(n,d)$ satisfying $n\le 25$ and $d\le 17$. For such pairs of $(n,d)$, we classify $X_F$ with $|\Aut(F)|\ge |\Aut(F_{d}^n)|$ in three steps:  
	\begin{enumerate}
		\item[(1)] $\Aut(F)$ are primitive linear groups (Theorem \ref{thm:prim}); 
		\item[(2)] $\Aut(F)$ are imprimitive linear groups (Theorem \ref{thm:imprim}); 
		\item[(3)] $\Aut(F)$ are reducible linear groups (Theorem \ref{thm:reducible}). 
	\end{enumerate}
	In these steps, our classification is based on known classifications of finite primitive linear groups in small degrees (see Section \ref{ss:prim}) and close relations among smoothness and shape of the defining polynomial $F$, canonical bounds, and Fermat-test ratios (see Section \ref{ss:refinedbounds}).
	Especially, we show that the existence of monomials in $F$ closely related to different primitive constituents of $\Aut(F)$ considerably reduce the order $|\Aut(F)|$ (Lemmas \ref{lem:type2bound}, \ref{lem:classify}, \ref{lem:d1d2}, \ref{lem:boundtypeII}), which plays a key role in steps (2) and (3).
	
	After our paper appeared on arXiv, Louis Esser and Jennifer Li informed us that they are now obtaining a similar result to our Theorem \ref{thm:Main} but in a different method \cite{EL24}.  Both results rely on Collins’ work on bounding the size of finite primitive complex linear groups, though the approaches employed are otherwise different.
	
	We conclude the introduction by posing a question closely related to our work. Collins \cite[Theorems A and A$^\prime$]{Col08b} obtained the optimal values of modular analogues of Jordan constants
	for finite linear groups. Dolgachev--Duncan \cite[Theorem 1.1]{DD19} classified all possible automorphism groups of smooth cubic surfaces over an algebraically closed field of arbitrary characteristic. It would be interesting to study an analogue of Theorem \ref{thm:Main} in positive characteristic. 
	
	\begin{question}
		Let $k$ be an algebraically closed field of characteristic $p>0$ and let integers $n\ge 1$, $d\ge 3$ with $(n,d)\neq (1,3), (2,3), (2,4)$. Classify smooth hypersurfaces $X\subset \P_k^{n+1}$ of degree $d$ with maximum $|\Aut(X)|$.
	\end{question}
	
	\subsection*{Acknowledgements} We would like to thank Professors Keiji Oguiso and Zhiwei Zheng for valuable conversations. We would also like to thank the referees for helpful comments and suggestions. This work is partially supported by the National Natural Science Foundation of China (No. 12171351, No. 12071337, No. 11921001). 
	
	\section{Notation and conventions}\label{ss:notation}
	\noindent
	(2.1)\, Let $W$ be a complex vector space of dimension $r\ge 1$. Let $W^*$ denote the dual vector space of $W$ and $S(W^*)$ be the symmetric algebra of $W^*$. Each element in $S(W^*)$ can be viewed as a complex-valued function on $W$. Then $\GL(W)$ acts from the left (resp. right) on $W$ (resp. $S(W^*)$) via $(f, w)\mapsto f(w)$ (resp. $(F,f)\mapsto f(F)$). Here $f(F)$ is given by $f(F)(w)=F(f(w))$ for any $w\in W$. Note that the action of $\GL(W)$ on $W$ induces an action of $\GL(W)$ on $\P(W)$ given by $(f,[w])\mapsto [f(w)]$. Two elements $F,F'$ in $S(W^*)$ are called {\it isomorphic} to each other if they lie in the same orbit under the action of $\GL(W)$ on $S(W^*)$. 
	
	\medskip
	
	\noindent
	(2.2)\, If we choose a basis $(e_1,\ldots,e_r)$ of $W$, we have the dual basis $(e_1^*,\ldots,e_r^*)$ of $W^*$. Then elements $f\in \GL(W)$ (resp. $w\in W$, resp. $F\in S(W^*)$) may be naturally viewed as matrices $A=(a_{ij})\in\GL(r,\C)$ (resp. vectors $(w_1,\ldots,w_r)\in \C^r$, resp.  polynomials $F(x_1,\ldots,x_r)\in \C[x_1,\ldots,x_r]$). Under such identifications, the actions of $\GL(r,\C)$ on $\C^r$ and $\C[x_1,\ldots,x_r]$ mentioned above are given by $$(A, (w_1,\ldots,w_r))\mapsto \big ( \sum_{i=1}^{r}a_{1i}w_i,\ldots,\sum_{i=1}^{r}a_{ri}w_i\big ) \, \text{ and }\, (F,A)\mapsto A(F):=F\big (\sum_{i=1}^{r}a_{1i}x_i,\ldots,\sum_{i=1}^{r}a_{ri}x_i\big )$$ respectively.  We call $(e_1,\ldots,e_r)$ (resp. $(e_1^*,\ldots,e_r^*)$) {\it the underlying basis} of $\C^r\cong W$ (resp. ${\C^r}^*\cong W^*$). This means that we identify the $e_i$'s with the standard basis of $\C^r$. In particular, in this paper, we always assume that the underlying bases of $W$ and $W^*$ are dual to each other, and if we choose a new basis $(e_i^\prime)$ as the underlying basis of $W$,  it is understood that the underlying basis of $W^*$ is changed to the dual basis $({e_i^\prime}^*)$. 
	
	\medskip
	
	\noindent
	(2.3) \,We say the underlying basis $(e_i)$ of $W$ is {\it compatible with} a decomposition $W=W_1\oplus\cdots \oplus W_s$ if $e_{n_1+\cdots+n_{i-1}+j}\in W_i$ for $1\le i\le s$, $1\le j\le n_i$, where $n_i={\rm dim}(W_i)$ and $n_0=0$. In this case, for any integer $d\ge 0$, the space of $d$-forms $S^d(W^*)\subset S(W^*)$ admits a decomposition $$S^d(W^*)=\Oplus_{k_1+k_2+\cdots+k_s=d}S^{k_1}(W_1^*)\otimes S^{k_2}(W_2^*)\otimes\cdots\otimes S^{k_s}(W_s^*),$$ 
	which induces a decomposition $$F=\sum_{k_1+k_2+\cdots+k_s=d} F^{(k_1,k_2,\ldots,k_s)}$$
	for any $d$-form $F\in S^d(W^*)$. We call $F^{(k_1,k_2,\ldots,k_s)}$ the {\it $S^{k_1}(W_1^*)\otimes S^{k_2}(W_2^*)\otimes\cdots\otimes S^{k_s}(W_s^*)$-component} of $F$.
	
	\medskip
	
	\noindent
	(2.4)\, Note that if $A_1,A_2\in \GL(r,\C)$, we have $(A_1A_2)(F)=A_2(A_1(F))$. We often use the notions of forms and homogeneous polynomials interchangeably. The {\it automorphism group} $\Aut(F)$ of a homogeneous polynomial $F\in\C[x_1,\ldots,x_r]$ of degree $d\ge 1$ is given by $$\Aut(F):=\{A\in \GL(r,\C)\mid\, A(F)=F\}.$$
	We say $F$ is {\it smooth} if $(0,\ldots,0)$ is the only solution for $\frac{\partial F}{\partial x_1}=\cdots=\frac{\partial F}{\partial x_r}=0$. In particular, if $r\ge 3$, a smooth form $F$ defines a smooth hypersurface $X_F\subset\P^{r-1}$ of dimension $r-2$. 
	We denote by $\pi: \GL(r,\C)\rightarrow \PGL(r,\C)$ the projection map. For a subgroup $H$ of $\GL(r,\C)$ (resp. $\PGL(r,\C)$), we say $H$ {\it preserves} $F$ if $H=G$ (resp. $H=\pi(G)$) for some subgroup $G\subset \Aut(F)$. In this case, we call $F$ an {\it $H$-invariant form} of degree $d$. We say a monomial $\mathfrak{m}$ of degree $d$ is in $F$ (or $\mathfrak{m}\in F$) if the coefficient of $\mathfrak{m}$ is not zero in the expression of $F$. As in \cite[Definition 4.4]{OY19}, for a finite subgroup $H_1\subset \PGL(r,\C)$, we say $H_1$ admits {\it a lifting} (resp. {\it an $F$-lifting}) if there exists $\widetilde{H}_1\subset \GL(r,\C)$ (resp. $\widetilde{H}_1\subset  \Aut(F)$) such that  $\widetilde{H}_1\cong \pi (\widetilde{H}_1)$ and $\pi (\widetilde{H}_1)=H_1$.
	
	\medskip
	
	\noindent
	(2.5)\, If $F_1,\ldots,F_k\in \C[x_1,\ldots,x_r]$ and $a>0$, we use $(F_1,\ldots,F_k)^a\cdot \C[x_1,\ldots,x_r]$ to denote the $a$-th power of the ideal in $\C[x_1,\ldots,x_r]$ generated by $F_1,\ldots,F_k$. For a group $G_1$, we denote by $Z(G_1)$ its center. We use $N.H$ to denote a finite group which fits in a short exact sequence $1\rightarrow N\rightarrow N.H\rightarrow H\rightarrow 1.$
	Some symbols frequently used in this paper are as follows:
	\begin{flushleft}
		\begin{tabular}{rl} 
			$\xi_k$ & the primitive $k$-th root $e^{\frac{2\pi i}{k}}$ of unity, where $k$ is a positive integer;\\
			
			$I_n$ & the identity matrix of rank $n$; \\
			
			$\mathcal{I}_{r,d}$ & the subgroup of $\GL(r,\C)$ of order $d$ generated by the scalar matrix $\xi_d I_r$;\\
			
			$C_n$ & the cyclic group of order $n$;\\
			
			$S_n$ & the symmetric group of degree $n$;\\
			
			$A_n$ & the alternating group of degree $n$;\\
			
			$Q_n$ & the quaternion group of order $n$. 
		\end{tabular} 
	\end{flushleft}	
	\section{Abelian subgroups and canonical bounds}\label{ss:abcan}
	In this section, we introduce two associated exact sequences \eqref{eq:1associ}-\eqref{eq:2associ} for any finite linear group using primitive groups in \S \ref{ss:primandasso}. In \S \ref{ss:3.2}, we show an upper bound for the size of abelian subgroups of automorphism groups $\Aut(F)$ of smooth forms $F$ (Theorem \ref{thm:diag}), introduce the notion of canonical bound for $|\Aut(F)|$ (Definition \ref{def:canonical} and Lemma \ref{lem:canonical}), and prove that Fermat forms have the largest automorphism groups among smooth forms with automorphism groups consisting of semi-permutations (Theorem \ref{thm:mono}).
	\subsection{Primitive constituents and associated exact sequences}
	
	\label{ss:primandasso}
	First, we recall the notion of primitive linear groups (for more details, see e.g., \cite{Bli17,Col07,Col08}). 
	\begin{definition}
	Let $\rho: G\rightarrow \GL(r,\C)$ be an irreducible linear representation of a finite group $G$. We say $\rho$ is {\it primitive} (resp. {\it imprimitive}) if $\C^r$ cannot (resp. can) be decomposed as a direct sum of proper subspaces permuted under the action of $G$. A finite subgroup in $\GL(r,\C)$ is called {\it primitive} (resp. {\it irreducible}) if the underlying representation is primitive (resp. irreducible). We call a finite subgroup $H\subset \PGL(r,\C)$ {\it primitive} if $H=\pi(\widetilde{H})$ for some finite primitive subgroup $\widetilde{H}\subset \GL(r,\C)$.
	\end{definition}
	A normal abelian subgroup of a finite primitive linear group $G$ is contained in the center $Z(G)$ of $G$ (see \cite[Lemma 1]{Col08}). Let $\widetilde{H}$ be a finite primitive subgroup in $\GL(r,\C)$ and $H:=\pi(\widetilde{H})\subset\PGL(r,\C)$. Since the center $Z(\widetilde{H})\subset \widetilde{H}$ is a subgroup of $Z(\GL(r,\C))=\{\lambda I_r\mid \lambda\in \C^\times\}$, we have $H\cong \widetilde{H}/Z(\widetilde{H})$. Using primitive groups, we introduce two exact sequences associated to finite linear groups. For this purpose and later use, we fix some notations.
	
	\begin{setup}\label{setup}
		Let $G$ be a finite subgroup of $\GL(r,\C)$ with $r\ge 1$. Then $\C^r$ admits a decomposition $\C^r=V_1\oplus \cdots \oplus V_m$ as a direct sum of irreducible $G$-stable subspaces for some integer $m\ge 1$. For each $1\le i\le m$, there is an integer $k_i\ge 1$ and a decomposition $V_i=W_{i1}\oplus \cdots \oplus W_{ik_i}$ as a direct sum of subspaces permuted (transitively) under the action of $G$ such that $${\rm Stab}_G(W_{ij}):=\{g\in G\mid \, g(W_{ij})=W_{ij}\}$$ acts primitively (but not necessarily faithfully) on $W_{ij}$ (see \cite[Lemma 1]{Col07}). Such permutations form transitive subgroups $K_{i}$ ($i=1,\ldots,m$) of $S_{k_i}$, and there is a natural group homomorphism $\psi_i: G\rightarrow K_i$. Let $s:=k_1+\cdots+k_m$. 
		The dimensions $r_{ij}$ of the subspaces $W_{ij}$ are called {\it subdegrees} of $G$. Let $$H_{ij}:=\{[h]\mid\, h=g|_{W_{ij}},g\in{\rm Stab}_G(W_{ij}) \}\subset \PGL(W_{ij}).$$ We call the $H_{ij}$'s the {\it primitive constituents} of $G$ (of degree $r_{ij}$) {\it belonging to} $V_i$. 
	\end{setup} 
	
	\begin{definition}\label{def:associ}
		Under $\bf{Set\textrm{-}up}$ \ref{setup}, we call the intersection $$P:=\bigcap_{1\le i\le m, 1\le j\le k_i} {\rm Stab}_G(W_{ij})$$ the {\it principal subgroup} of $G$. Consider the group homomorphism
			$$\psi: G\longrightarrow K_{1}\times\cdots\times K_{m},\;\; g\mapsto (\psi_1(g),\ldots,\psi_m(g)).$$
			Note that ${\rm Ker}(\psi)=P$.
		Then we call the following exact sequence the {\it first associated} exact sequence of $G$
			\begin{equation}\label{eq:1associ}
				1\longrightarrow P\longrightarrow G\xrightarrow{\,\;\psi\;\,} K_1\times\cdots\times K_m.
		\end{equation}
		Let $\widetilde{H}_{ij}:=\{h\mid\, h=g|_{W_{ij}},g\in{\rm Stab}_G(W_{ij}) \}$. Then we have surjective group homomorphisms
		$\phi_{ij}: {\rm Stab}_G(W_{ij})\rightarrow \widetilde{H}_{ij}$, $g\mapsto g|W_{ij}$
		and
		$\pi_{ij}: \widetilde{H}_{ij}\rightarrow H_{ij}$, $h\mapsto [h].$
		Thus, $\widetilde{H}_{ij}$ and $H_{ij}$ are primitive subgroups of $\GL(W_{ij})$ and $\PGL(W_{ij})$ respectively.  Consider the group homomorphism $$\phi: P\longrightarrow H_{11}\times\cdots\times H_{1k_1}\times\cdots\times H_{mk_m},\;\; g\mapsto ([g|W_{11}],\ldots,[g|W_{1k_1}],\ldots,[g|W_{mk_m}]).$$
		Let $N:={\rm Ker}(\phi)$. Then we call the following exact sequence the {\it second associated} exact sequence of $G$ 
		\begin{equation}\label{eq:2associ}
				1\longrightarrow N\longrightarrow P\xrightarrow{\,\;\phi\;\,} H_{11}\times\cdots\times H_{1k_1}\times\cdots\times H_{mk_m}.
		\end{equation}
	\end{definition} 
	
	Clearly, the decomposition of the $V_i$'s as in $\bf{Set\textrm{-}up}$ \ref{setup} guarantees the existence of the first and second associated exact sequences, which is a key point of our approach.
	After choosing a basis for each $W_{ij}$, we obtain a basis of $\C^r$ and we may identify $\GL(W_{ij})$ (resp. $\PGL(W_{ij})$) with $\GL(r_{ij},\C)$ (resp. $\PGL(r_{ij},\C)$). Then each element $A$ in $P$ can be viewed as a block diagonal matrix $A={\rm diag}(A_{11},\ldots,A_{mk_m})$ with blocks $A_{ij}\in \GL(r_{ij},\C)$.  Note that $N$ is abelian and its elements are of the form ${\rm diag}(\lambda_{11}I_{r_{11}},\lambda_{12}I_{r_{12}},\ldots,$ $\lambda_{mk_m}I_{r_{mk_m}})$, where $\lambda_{ij}$ are non-zero complex numbers.
	
	\begin{definition}\label{def:primconssequ}
		Under $\bf{Set\textrm{-}up}$ \ref{setup}, let $l(G):=(r_1^\prime,\ldots,r_s^\prime)$ be the  $s$-tuple of positive integers such that $r_1'\ge r_2'\ge\cdots \ge r_s'$ and there exists a permutation $\sigma\in S_s$ satisfying $(r_{\sigma(1)}',\ldots,r_{\sigma(s)}')=(r_{11},r_{12},\ldots,r_{mk_m})$. We call $l(G)$ the {\it subdegree sequence} of $G$ (of length $s$) and the $s$-tuple $(H_{11},H_{12},\ldots,H_{m(k_{m}-1)},H_{mk_m})$ a {\it primitive-constituent sequence} of $G$.  The $m$-tuple of integers $(k_1,\dots,k_m)$ is called an  {\it intrinsic multiplicity sequence} of $G$.   
	\end{definition}
	
		Throughout the paper, for any finite subgroup of  $\GL(r, \C)$, we fix a primitive-constituent sequence once and for all.
	
	\subsection{Abelian subgroups and canonical bounds}\label{ss:3.2}
	It is known that for a smooth homogeneous polynomial $F(x_1,\dots,x_r)$ of degree $d\ge 3$, every abelian subgroup of $\Aut(F)$ has order at most $d^{r}$ (see \cite[Lemma 3.1]{HS80}, \cite[Abelian Lemma]{Sza96}). The following more refined bound
	for abelian subgroups of the automorphism groups of smooth forms plays a key role in our study.
	
	\begin{theorem}\label{thm:diag}
		Let $F=F(x_1,\dots,x_{r})$ be a smooth homogeneous polynomial of degree $ d\geq3$, where $r\geq2$. Fix positive integers $r_1,\dots,r_m$ with $\sum_{1\leq i\leq m}{r_i}=r$, where $1\le m\le r$. Let $N\subset \Aut(F)$ be an abelian subgroup generated by matrices of the form $\Diag(\lambda_1 I_{r_1}, \lambda_2 I_{r_2},\dots,\lambda_m I_{r_m})$. Then $|N|\le d^m$.
	\end{theorem}
	
	\begin{proof}
		We prove the theorem by adapting the proof of \cite[B\'{e}zout Lemma]{Sza96}. If $r=2$, then $|N|\le d^m$ by direct calculation. If $r\ge 3$, then $X:=X_F\subset \P^{r-1}$ is a smooth hypersurface of degree $d$. Fix $x=(a_1:\cdots :a_r)\in X$ with $a_i\neq 0 $ for all $i$. Let $T\subset \PGL(r,\C)$ be the (abelian) subgroup generated by all invertible diagonal matrices. The stabilizer $T_x$ of $x$ is trivial and we identify $T$ with the orbit $Tx$ of $x$ via the $T$ action. We set $U=X\cap T$ and we have ${\rm Lin}(X)\cap T=\bigcap_{s\in U}\, Us^{-1}$. Let $H\subset T$ be the subgroup generated by all matrices of the form $\Diag(\lambda'_1 I_{r_1}, \lambda'_2 I_{r_2},\dots,\lambda'_m I_{r_m})$, where $\lambda'_i\in \C^\times$ ($1\le i\le m$). Let $Y\subset \P^{r-2}$ denote the closure of the orbit $Hx$ of $x$. Note that $Y\cong \P^{m-1}$. Then ${\rm Lin}(X)\cap H$=${\rm Lin}(X)\cap T\cap H$=$\bigcap_{s\in U}\, (Us^{-1}\cap H)$. The cardinality of $\bigcap_{s\in U}\, (Us^{-1}\cap H)$ is at most that of $\bigcap_{s\in U}\, (Us^{-1}\cap Y)$. The finiteness of $\bigcap_{s\in U}\, (Us^{-1}\cap Y)$ implies that for $m-1$ general points $s_0,s_1,...,s_{m-2}$, the intersection $\bigcap_{i=0}^{m-2}\, (Us_i^{-1}\cap Y)$ is already finite. The cardinality of this intersection is at most ${\rm deg}(\bigcap_{i=0}^{m-2}(X s_i^{-1} \cap Y))$. By applying B\'{e}zout's Theorem (\cite[Example 8.4.6]{Ful98}), we obtain $|{\rm Lin}(X)\cap H|\le d^{m-1}$. Then $|N|\le d\cdot |\pi(N)|\le d^m$ since $\pi(N)\subseteq ({\rm Lin}(X)\cap H)$.\end{proof}
	The proof, using the ideas in \cite{CS95} and \cite{Sza96}, is suggested by a referee. A longer proof of Theorem \ref{thm:diag} without using B\'{e}zout's Theorem can be found in \cite[Theorem 3.4]{YYZ24}. 
	Now we are ready to introduce a bound for the size of groups of automorphisms of smooth forms.
	
	\begin{definition}\label{def:canonical}
		Let $G$ be a subgroup of the automorphism group $\Aut(F)$ of a smooth form $F=F(x_1,\ldots,x_r)$ of degree $d\ge 3$, where $r\ge 1$. Let $(H_1,\dots, H_s)$ be a primitive-constituent sequence of $G$ with an intrinsic multiplicity sequence $(k_1,\dots,k_m)$. We define the {\it canonical bound} of $G$ by
		$$B(G,F)=d^s\cdot\prod_{i=1}^{s}|H_i|\cdot\prod_{j=1}^{m}{k_j !}.$$
		To simplify notation, we often use $B(G)$ to denote $B(G, F)$ if there is no confusion.
	\end{definition}
	
	Combining Theorem \ref{thm:diag} with the two associated exact sequences \eqref{eq:1associ} and \eqref{eq:2associ} of finite linear groups, we have the following observation.
	
	\begin{lemma}\label{lem:canonical}
		Let $F=F(x_1,\ldots,x_r)$ be a smooth form of degree $d\ge 3$, where $r\ge 1$. If $G$ is a subgroup of $\Aut(F)$, then $|G|\le B(G,F)$.
	\end{lemma}
	
	The following example gives an illustration for computing canonical bounds.
	
	\begin{example}\label{ex:canonical}
		Following \cite[Example 2.1 (17)]{OY19}, let $F=\left(\left(x_1^4+x_2^4\right)+\left(2+4\xi_3^2\right)x_1^2x_2^2\right)x_3+\left(-\left(x_1^4+x_2^4\right)+\left(2+4\xi_3^2\right)x_1^2x_2^2\right)x_4+x_3^4x_4+x_4^4x_3+x_5^5$ and the subgroup $\widetilde{G}$ of $\Aut(F)$
		$$\widetilde{G}=\left\langle 
		\begin{pmatrix} 
			\xi_8^3&0&0&0&0 \\ 
			0&\xi_8&0&0&0 \\ 
			0&0&0&1&0 \\ 
			0&0&1&0&0 \\ 
			0&0&0&0&1 \\ 
		\end{pmatrix},\begin{pmatrix} 
			-\frac{1}{\sqrt{2}}\xi_8&\frac{1}{\sqrt{2}}\xi_8&0&0&0 \\ 
			\frac{1}{\sqrt{2}}\xi_8^3& \frac{1}{\sqrt{2}}\xi_8^3&0&0&0 \\ 
			0&0&\xi_3&0&0 \\ 
			0&0&0&\xi_3^2&0 \\ 
			0&0&0&0&1 \\ 
		\end{pmatrix},\begin{pmatrix} 
			-1&0&0&0&0 \\ 
			0&1&0&0&0 \\ 
			0&0&1&0&0 \\ 
			0&0&0&1&0 \\ 
			0&0&0&0&\xi_5  \\ 
		\end{pmatrix} \right\rangle.$$
		We denote by $(e_1,e_2,e_3,e_4,e_5)$ the  of $\C^5$. Then using notations in $\bf{Set\textrm{-}up}$ \ref{setup}, we have  $\C^5=V_1\oplus V_2\oplus V_3$, $V_1=W_{11}$, $V_2=W_{21}\oplus W_{22}$, $V_3=W_{31}$, where $V_1=\langle e_1, e_2\rangle$, $V_2=\langle e_3, e_4\rangle$, $V_3=\langle e_5\rangle$, $W_{21}=\langle e_3\rangle$, $W_{22}=\langle e_4\rangle$. The linear group $\widetilde{G}$ has $4$ primitive constituents $H_{ij}$. Moreover, $S_4\cong H_{11}\subset \PGL(2,\C)$, $|H_{21}|=|H_{22}|=|H_{31}|=1$, the subdegree sequence $l(\widetilde{G})=(2,1,1,1)$, and $\widetilde{G}$ has an intrinsic multiplicity sequence $(1,2,1)$. Then we have $B(\widetilde{G},F)=5^4\cdot 24\cdot 2=30000>480=|\widetilde{G}|$, as predicted by Lemma \ref{lem:canonical}.
	\end{example}
	
	We say an invertible matrix $A$ is {\it semi-permutation} if $A$ is a diagonal matrix up to permutation of columns. Note that up to linear change of coordinates, a finite linear group consists of semi-permutation matrices if its subdegree sequence is $(1,\dots,1)$. Using Theorem \ref{thm:diag}, we obtain an optimal bound for the size of semi-permutation automorphism groups of smooth forms.
	
	\begin{theorem}\label{thm:mono} 
		Let  $F=F(x_1,\dots,x_r)$ be a smooth form of degree $d\geq3$ and $r\geq2$. If $\Aut(F)$ consists of semi-permutations, then $|\Aut(F)|\leq d^{r}\cdot r!$. The equality occurs if and only if $F$ is isomorphic to the Fermat form $F_{d}^{r-2}=x_1^d+x_2^d+\cdots+x_r^d$.
	\end{theorem}
	
	\begin{proof}
		We identify $S_r$ with the group of all permutations of the coordinates $x_1,...,x_r$. Let $N_1\subseteq \Aut(F)$ be the (abelian) subgroup consisting of all diagonal matrices in $\Aut(F)$. Since $\Aut(F)$ only has semi-permutations, $N_1$ is  normal in $\Aut(F)$ and the quotient group $Q_1:=\Aut(F)/N_1$ is isomorphic to a subgroup of $S_r$. Then by Theorem \ref{thm:diag}, we have $|\Aut(F)|\le d^{r} \cdot r!$ and the equality holds only if $|N_1| =d^{r}$ and  $Q_1\cong S_{r}$. We assume now $|\Aut(F)|= d^{r} \cdot r!$. If $x_1^d\in F$, then $x_i^d\in F$ for all $i\in\{1,2,...,r\}$ by $Q_1\cong S_{r}$. Thus, for every $A={\rm diag}(\lambda_1,...,\lambda_r)\in N_1$, we have $\lambda_i^d=1$ for all $i$ by $A(F)=F$. Then $N_1=\{{\rm diag}(\lambda_1,...,\lambda_r)|\, \lambda_1^d=\cdots=\lambda_r^d=1\},$ which implies that $\{x_i^d\}_{1\leq i\leq r}$ are the only monomials preserved by $N_1$ and $F$ is a Fermat form, up to replacing $x_i$ by their non-zero multiples. If $x_1^d\notin F$, then by \cite[Proposition 3.3]{OY19} and $Q_1\cong S_{r}$, we have $x_i^{d-1}x_j\in F$ for all $i,j$ with $i\neq j$. Then $N_1\subseteq N_2:=\{{\rm diag}(\lambda_1,...,\lambda_r)|\, \lambda_i^{d-1}\lambda_j=1 \text{ for all $i,j$ with $i\neq j$} \},$
		which is a contradiction since $d^r=|N_1|\le |N_2|<d^r$ by direct computation.
	\end{proof}
	
	\section{Fermat-test ratios and proof of most cases}\label{ss:FTR}
	A famous theorem of Jordan \cite{Jor78} states that, for every integer $r\geq1$,  there exists a constant  such that any finite subgroup of $\GL(r,\C)$ has an abelian normal subgroup of index at most this constant. Consequently, for each integer $r\ge 1$, there are only finitely many finite primitive groups in $\PGL(r,\C)$ (up to conjugation), and Collins \cite[Theorem A]{Col08} determined the maximum order ${\rm JC}(r)$ among such groups. Note that $\JC(r)={\rm J}(r)$ for $r\geq 71$, but $\JC(r)$ and ${\rm J}(r)$ are often different for $r<71$ (see \cite[Theorems A, B and D]{Col07}). In \S \ref{ss:4.1}, based on canonical bounds and the values of ${\rm JC}(r)$, we introduce the notion of Fermat-test ratio (Definition \ref{def:FTR}) and we show that Fermat-test ratios have many nice properties which will be frequently used in our study. In \S \ref{ss:4.2}, using Fermat-test ratios, we quickly prove that our main result Theorem \ref{thm:Main} holds if $n\ge 26$ or $d\ge 18$ (Theorem \ref{thm:boundnd}).
	
	  \subsection{Fermat-test ratios}\label{ss:4.1}
		As a generalization of subdegree sequences of linear groups, the following concepts will be used in the definition of Fermat-test ratios.
	
	\begin{definition}\label{def:subseq}
		Let $s$ be a positive integer. We call an $s$-tuple $l=(r_1^\prime,\dots,r_s^\prime)$ of positive integers $r_i^\prime$ a {\it subdegree sequence} of length $s$ if $r_1^\prime\ge r_2^\prime\ge\cdots\ge r_s^\prime$. The integers $r_i^\prime$ are called ${\it subdegrees}$ of $l$.  The {\it total degree} $v(l)$ of $l$ is defined by $v(l)=r_1^\prime+\cdots+r_s^\prime$.  Suppose $l$ has exactly $m$ distinct subdegrees $r_1>r_2>\cdots>r_m$, where $m\ge 1$. We call  $r_1^{k_1} r_2^{k_2}\cdots r_m^{k_m}$ the {\it exponential type} of $l$, where $k_i$ are the multiplicities of $r_i$ in $l$. We call $(k_1,\ldots,k_m)$ the {\it multiplicity sequence} of $l$. We often denote $l$ by its exponential type if there is no confusion. Clearly, we have two equations: $v(l)=\sum_{i=1}^{m} r_i k_i$, $s=\sum_{i=1}^{m}k_i$. 
		For a subdegree sequence $l=(r_1^\prime,\ldots,r_s^\prime)$, if $\mathcal{H}=(H_1,\ldots,H_s)$ is an $s$-tuple of finite primitive subgroups $H_i\subset \PGL(r_i^\prime,\C)$, we call $l(\mathcal{H}):=l$  the {\it subdegree sequence} of $\mathcal{H}$. 
	\end{definition}
	In practice, we often use $\mathcal{H}$ to denote a primitive-constituent sequence.
	
		\begin{example}
		Let $\widetilde{G}$ be as in Example \ref{ex:canonical}. Then the subdegree sequence $l(\widetilde{G})=(2,1,1,1)$ has the exponential type $2^1 1^3$, and the multiplicity sequence of $l(\widetilde{G})$ is $(1,3)$ which is different from the intrinsic multiplicity sequence $(1,2,1)$ of $\widetilde{G}$. The total degree $v(l(\widetilde{G}))$ is 5.
	\end{example}	
	
	\begin{remark}
		We allow $r^0$ to appear in the exponential type of $l$ even if $r$ is not a subdegree of $l$. Under such convention, the equations in Definition \ref{def:subseq} still hold. For example, by $l$ being the subdegree sequence of the exponential type $5^03^22^01^3$ we mean $l=(3,3,1,1,1)$.
	\end{remark}
	
	For later use, we define sum/difference between subdegree sequences.
	
	\begin{definition}
		Let $l_1$ and $l_2$ be two subdegree sequences of exponential types $r_1^{a_1} r_2^{a_2}\cdots r_m^{a_m}$ and $r_1^{b_1} r_2^{b_2}\cdots r_m^{b_m}$, where $a_i,b_i\ge 0$, $m\ge 1$, $r_1>\cdots >r_m\ge 1$. We denote by $l_1+l_2$ the subdegree sequence of the exponential type $r_1^{a_1+b_1} r_2^{a_2+b_2}\cdots r_m^{a_m+b_m}$. If $a_i\ge b_i$ for all $i$, we denote by $l_1-l_2$ the subdegree sequence of the exponential type $r_1^{a_1-b_1} r_2^{a_2-b_2}\cdots r_m^{a_m-b_m}$.
	\end{definition}	
	
	We recall the values of the function ${\rm JC}(r)$ (\cite[Theorem A]{Col08}). For $r=10,11$ or $r\geq13$, $\JC(r)=(r+1)!$; for the remaining cases, ${\rm JC}(r)$ is as follows: 
	\begin{table}[H]\rm
		\renewcommand\arraystretch{1.1}
		\begin{tabular}{ccccccccccc}
			\hline
			$r$&1&2&3&4&5&6&7&8&9&12\\
			\hline
			${\rm JC}(r)$&1&60&360&25920&25920&6531840&1451520&348364800&4199040&448345497600\\
			\hline
		\end{tabular}
	\end{table}
	\noindent
	Now we give the definition of Fermat-test ratios which is crucial for our classification of smooth hypersurfaces with large automorphism groups.
	
	\begin{definition}\label{def:FTR}
		Let $l=(r_1^\prime,\ldots,r_s^\prime)$ be a subdegree sequence of length $s\ge 1$ with the exponential type $r_1^{k_1}r_2^{k_2}\cdots r_m^{k_m}$ and total degree $r\ge 1$. For an integer $d\ge 3$, we define the {\it Fermat-test ratio} of the pair $(l,d)$ by
		\begin{equation*}
			R(l,d)=\frac{d^s\cdot\prod\limits_{i=1}^{s}{{\rm JC}(r_i^\prime)}\cdot\prod\limits_{j=1}^{m}{k_j!}}{d^{r}\cdot r!}.
		\end{equation*}
		Let $\mathcal{H}=(H_1,\ldots,H_s)$ be an $s$-tuple of finite primitive subgroups $H_i\subset \PGL(r_i^\prime,\C)$ with the subdegree sequence $l(\mathcal{H})=l$. We define the {\it Fermat-test ratio} of the pair $(\mathcal{H},d)$ by
		$$R(\mathcal{H},d)=\frac{d^s\cdot\prod\limits_{i=1}^{s}{|H_i|}\cdot\prod\limits_{j=1}^{m}{k_j!}}{d^{r}\cdot r!}.$$
		Let $F=F(x_1,\ldots,x_r)$ be a smooth form of degree $d\ge 3$, where $r\ge 1$. We call $R(l(\Aut(F)),d)$ the {\it Fermat-test ratio of $F$}, denoted by $R(F)$.
	\end{definition}
	
		The following simple fact is useful for our study. 
	
	\begin{lemma}\label{lem:factorial}
		Let $k_1,\ldots,k_m$ be positive integers with $m\ge 1$. Then $\prod_{j=1}^{m}{k_j !}\le (k_1+\cdots+k_m)!.$
		The equality occurs if and only if $m=1$.
	\end{lemma}
	
	By Lemmas \ref{lem:canonical}, \ref{lem:factorial} and definition of Fermat-test ratios, we have the following observation.
	
	\begin{lemma}\label{lem:FTRandGF}
		Let $F=F(x_1,\ldots,x_r)$ be a smooth form of degree $d\ge 3$ and $r\ge 1$. Let $\mathcal{H}:=(H_1,\ldots,H_s)$ be a primitive-constituent sequence of $\Aut(F)$ with $r_1\geq\cdots \geq r_s$, where $r_i$ is the degree of $H_i$. Then we have 
		$$R(F)\geq \frac{R(F)\cdot |H_i|}{{\rm JC}(r_i)}\geq R(\mathcal{H},d)\geq \frac{B(\Aut(F))}{d^r\cdot r!}\ge \frac{|\Aut(F)|}{d^r\cdot r!}.$$
		In particular, if $|\Aut(F)|\geq d^{r}\cdot r!$, then $R(F)\geq R(\mathcal{H},d)\ge 1$.
	\end{lemma}
	
	By Definition \ref{def:FTR}, we find some basic properties of Fermat-test ratios.
	
	\begin{lemma}\label{lem:ratioproperties} 
		Let $l$ be a subdegree sequence of length $s\ge 1$. Let $d,d'$ be integers at least $3$. Then $R(l,d)/R(l,d')=(d'/d)^{v(l)-s}$.
		In particular, if $d>d'$, then $R(l,d)\le R(l,d')$.
	\end{lemma}
	
	The properties of Fermat-test ratios in Lemmas \ref{lem:ratioproperties2}, \ref{lem:addr0} and \ref{lem:boundl1} make computing/bounding values of Fermat-test ratios quite feasible.
	
	\begin{lemma}\label{lem:ratioproperties2}
		Let $l_1$, $l_2$ be subdegree sequences and let $d\geq3$ be a positive integer. Then we have
		\begin{enumerate}
			\item [(1)] $\tbinom{v({l_1+ l_2})}{v({l_1})}\cdot R(l_1+ l_2,d)\geq R(l_1,d)\cdot R(l_2,d)$. The equality occurs if and only if the subdegrees of $l_1$ and $l_2$ are disjoint;
			\item [(2)] Suppose the subdegrees of $l_1$ and $l_2$ are disjoint. Let $l_2'$ be a subdegree sequence. If $v(l_2)\geq v(l_2')$ and $R(l_2,d)\leq R(l_2',d)$, then $R(l_1+ l_2,d)\leq R(l_1+ l_2',d)$; and
			\item [(3)] Let $s_i$ be the length of $l_i$ ($i=1,2$). Suppose  $s_1\geq s_2$ and $v(l_1)\leq v(l_2)$. If $R(l_1,d)>R(l_2,d)$, then $R(l_1,d+1)>R(l_2,d+1)$.
		\end{enumerate}
	\end{lemma}
	\begin{proof}
			We only give the proof of case (1) and the other cases are similar. Suppose the exponential type of $l:=l_1+l_2$ is $r_1^{k_1}r_2^{k_2}\cdots r_m^{k_m}$. Then we may write $l_1=r_1^{k'_1}r_2^{k'_2}\cdots r_m^{k'_m}$ and $l_2=r_1^{k''_1}r_2^{k''_2}\cdots r_m^{k''_m},$ where $k'_i\geq 0$, $k''_i \geq 0$ and $k_i=k'_i+k''_i>0$.
    		From Lemma \ref{lem:factorial}, we conclude that $k_i!\geq k'_i! \cdot k''_i!$ and the equality holds if and only if $k'_i\cdot k''_i=0$. By Definition \ref{def:FTR}, we have $$R(l,d)=R(l_1,d)R(l_2,d)\cdot\frac{v(l_1)!\cdot v(l_2)!}{v(l_1+l_2)!}\cdot \prod_{i=1}^{m}\frac{ k_i!}{k_i'!\cdot k_i''!}\geq R(l_1,d)R(l_2,d)\cdot\frac{1}{\tbinom{v({l_1+ l_2})}{v({l_1})}},$$ and the equality holds if and only if $k'_i\cdot k''_i=0$ for all $i$ (i.e., the subdegrees of $l_1$ and $l_2$ are disjoint).
	\end{proof}
	\begin{lemma}\label{lem:addr0}
		Let $l$ be a subdegree sequence containing $r_0>1$ with multiplicity $k_0\ge 1$. Then, for $d\geq3$, $R(l,d)>R(l+(r_0),d)$  holds if one of the following conditions is true:
		\begin{enumerate}
			\item [(1)] $r_0=2$ and $k_0\geq 5$;
			\item [(2)] $r_0=4$ and $k_0\geq 2$;
			\item [(3)] $r_0=3$ or $r_0\geq 5$.
		\end{enumerate}
		In particular, if $l$ is of exponential type $r_0^{k_0}$ with $r_0\notin\{2,3,4,5,6,8\}$, then $R(l,d)<1$.
	\end{lemma}
	\begin{proof}
		Let $l_2$ be the subdegree sequence of the exponential type $r_0^{k_0}$. We set $l_1:=l-l_2$. By (1) in Lemma \ref{lem:ratioproperties2}
		we have
		$$\tbinom{v({l})+r_0}{(k_0+1)r_0}R(l+ (r_0),d)= R(l_1,d) R(l_2+(r_0),d),\;\; \tbinom{v({l})}{k_0r_0} R(l,d)= R(l_1,d) R(l_2,d).$$
		Then the ratio
		\begin{equation}\label{eq:decreasing}
			\lambda:=\frac{R(l+ (r_0),d)}{R(l,d)}=\frac{\tbinom{v({l})}{k_0r_0}}{\tbinom{v({l})+r_0}{(k_0+1)r_0}}\cdot\frac{R(l_2+ (r_0),d)}{R(l_2,d)}=\frac{v(l)!}{(v(l)+r_0)!}\cdot\frac{{\rm JC}(r_0)(k_0+1)}{d^{r_0-1}}.
		\end{equation}
		Since $\lambda$ decreases as $v(l_1)$ increases or $d$ increases, we may take $v(l_1)=0$ and $d=3$. Then 
		\begin{equation}\label{eq:lambdabound}
			\lambda \leq\frac{(r_0k_0)!}{(r_0(k_0+1))!}\cdot\frac{(k_0+1){\rm JC}(r_0)}{3^{r_0-1}}=\frac{(r_0k_0)!}{(r_0(k_0+1)-1)!}\cdot\frac{{\rm JC}(r_0)}{r_03^{r_0-1}}.
		\end{equation}
		We notice that the value in the right hand side of the inequality decreases as $k_0$ increases. By \eqref{eq:lambdabound}, if $r_0=2$, $k_0=5$, we have $\lambda\leq\frac{10!\cdot60}{11!\cdot2\cdot3}=\frac{10}{11}<1$. Thus (1) follows.
		Similarly, one may check if $(r_0,k_0)=(3,1)$, $(4,2)$, $(5,1)$, $(6,1)$, $(7,1)$, $(8,1)$, $(9,1)$, $(12,1)$, we have $\lambda<1$. 
		
		For $r_0=10,11$ or $r_0\geq 13$, we have $\lambda<\frac{(r_0+1)!}{3^{r_0-1}r_0^{r_0}}<\frac{r_0+1}{3^{r_0-1}}\cdot\frac{r_0^{r_0}}{r_0^{r_0}}<1$.
		By Lemma \ref{lem:ratioproperties} and Definition \ref{def:FTR}, we find that $R((r_0),d)\leq R((r_0),3)<1$ for $r_0=7$, or $r_0\geq 9$. Therefore, $R(r_0^{k_0},d)<1$ if $r_0=7$ or $r_0\ge 9$.
	\end{proof}
	
	\begin{remark}
		 For the proof of Lemma \ref{lem:boundl1}, we record the values of some Fermat-test ratios. $R(2^{13},3)=16000000000/12649365729$, $R(3^1,3)=20/3$, $R(3^3,3)=200/189$, $R(4^1,3)=40$, $R(4^2,3)=320/7$, $R(4^3,3)=2560/231$, $R(5^1,3)=8/3$, $R(6^1, 3)=112/3$, $R(6^{2},3)=896/297$, $R(8^1,3)=320/81$.
	\end{remark}

	\begin{lemma}\label{lem:boundl1}
		Let $d\ge 3$ and $m\ge 1$ be integers. Let $l_1$ be a subdegree sequence of exponential type $r_1^{k_1}\cdots r_m^{k_m}$ with $r_i>1$, $k_i>0$ and $R(r_i^{k_i},d)\ge1$ for all $1\le i\le m$. Then we have $R(l_1,d)\le R(l_1^\prime,d)$, where
		\begin{equation*}
			l_1^\prime \text{ is of exponential type}
			\begin{cases}
				4^1 & \quad \text{if } v(l_1)=4 ~~and~~d=3 \\
				2^a {~~with ~~}a=v(l_1)/2 & \quad \text{if } v(l_1){ ~~is~~even ~~and ~~}(v(l_1),d)\neq(4,3)\\
				2^a {~~with ~~}a=(v(l_1)-1)/2 & \quad \text{if } v(l_1){ ~~is~~odd.}\\
			\end{cases}
		\end{equation*}
		In particular, if $v(l_1)\geq28$, we have $R(l_1,d)<1$.
	\end{lemma}
	\begin{proof}
		By Lemma \ref{lem:addr0}, we have $r_i\in \{2,3,4,5,6,8\}$ for all $i$. Then we may assume $l_1$ is of exponential type $8^{a_8}6^{a_6}5^{a_5}4^{a_4}3^{a_3}2^{a_2}$ with $a_i\ge 0$. Furthermore, as $R(2^{14},3)$, $R(3^4,3)$, $R(4^4,3)$, $R(5^2,3)$, $R(6^3,3)$, and $R(8^2,3)$ are all less than 1, applying Lemmas \ref{lem:ratioproperties} and \ref{lem:addr0}, we have $a_8\le 1$, $a_6\le 2$, $a_5\le 1$, $a_4\le 3$, $a_3\le 3$, $a_2\le 13$. By calculations, we observe $R(5^1 3^1,3)<R(2^4,3)$, $R(5^1 3^3,3)<R(2^7,3)$, $R(3^1,3)<R(2^1,3)$, $R(3^2,3)<R(2^3,3)$, $R(3^3,3)<R(2^4,3)$, $R(4^2,3)<R(2^4,3)$, $R(4^3,3)<R(2^6,3)$, $R(5^1,3)<R(2^2,3)$, $R(6^1,3)<R(2^3,3)$, $R(6^2,3)<R(2^6,3)$ and $R(8^1,3)<R(2^4,3)$.
		
		If $a_4\neq 1$, then by Lemma \ref{lem:ratioproperties2} (2)-(3) and replacing subdegrees $3,4,5,6,8$ by $2$ with suitable multiplicities (e.g. replacing $5^13^1$ by $2^4$), we have $R(l_1,d)\leq R(2^{a_2^\prime},d)$ with $a_2^\prime=\frac{v(l_1)}{2}$ $\big (\text{resp. } a_2^\prime=\frac{v(l_1)-1}{2}\big )$ when $v(l_1)$ is even (resp. odd).
		
		If $a_4=1$, then similar to the previous case, we have $R(l_1,d)\leq R(4^1 2^{b_2},d)$ with $b_2=\frac{v(l_1)-4}{2}$ $\big (\text{resp. } b_2=\frac{v(l_1)-5}{2}\big )$ when $v(l_1)$ is even (resp. odd). Since the ratio $R(4^1 2^{b_2},d)/R(2^{b_2+2},d)$ $=\frac{25920}{3600 d(b_2+1)(b_2+2)}$, we find that $R(4^1 2^{b_2},d)\le R(2^{b_2+2},d)$ unless $(b_2,d)=(0,3)$. Thus, we have $R(l_1,d)\le R(2^{b_2+2},d)$ unless $(v(l_1),d)=(4,3)$.
		
		Moreover, since $R(2^{14},d)<1$ for $d=3$, by Lemmas \ref{lem:ratioproperties} and \ref{lem:addr0}, we have $R(l_1,d)<1$ for all $l_1$ satisfying $v(l_1)\geq28$.
	\end{proof}
	
	\begin{remark}\label{rem:Rvalues}
		For later use, here we present (in decreasing order) the values $R(l_1^\prime,3)$ for all possible $l_1^\prime$ with $v(l_1^\prime)\le 26$ in Lemma \ref{lem:boundl1}: $R(2^5,3)=20000/189$, $R(2^6,3)=200000/2079$, $R(2^4,3)=2000/21$, $R(2^7,3)=2000000/27027$, $R(2^3,3)=200/3$, $R(2^8,3)=4000000/81081$, $R(4^1,3)=40$, $R(2^2,3)=100/3$, $R(2^9,3)=40000000/1378377$, $R(2^{10},3)=400000000/26189163$, $R(2^1,3)=10$, $R(2^{11},3)=4000000000/549972423$, $R(2^{12},3)=40000000000/12649365729$, $R(2^{13},3)=16000000000/12649365729$.
	\end{remark}
	
    \subsection{Proof of most cases}\label{ss:4.2}
		Next we use Fermat-test ratios to prove that our main result Theorem \ref{thm:Main} holds for all possible pairs $(n,d)$ except only finitely many cases.
	
	\begin{theorem}\label{thm:boundnd}
		Let  $F=F(x_1,\dots,x_{n+2})$ be a smooth form of degree $d$, where $n\geq1$, $d\geq3$. Suppose that $X_F$ is not isomorphic to the Fermat hypersurface in $\P^{n+1}$ of degree $d$.
		If $n\geq 26$ or $d\geq18$, then $|\Aut(X_F)|< d^{n+1}\cdot(n+2)!$.
	\end{theorem}
	\begin{proof}
		Suppose the subdegree sequence $l:=l(\Aut(F))$ of $\Aut(F)$ is of exponential type $r_1^{k_1}\cdots r_m^{k_m}$ with $m\ge 1$ and $k_i>0$ for all $i$. By Theorem \ref{thm:mono}, we may assume $r_1>1$. Then by Lemma \ref{lem:FTRandGF}, it suffices to show $R(l,d)<1$. Let $\mathcal{A}:=\{i| \; r_i>1, R(r_i^{k_i},d)\ge 1, 1\le i\le m\}.$ We define $l_1$ to be the sum of all the subdegree sequences of exponential types $r_j^{k_j}$ ($j\in \mathcal{A}$). Let $l_2:=l-l_1$.
		
		Case $n\ge 26$. If $l=l_1$ or $l=l_2$, then by Lemmas \ref{lem:boundl1} and \ref{lem:ratioproperties2} (1), $R(l,d)<1$. Then we may assume that $l\neq l_1,l_2$ and $v(l_1)<28$. Since $n+2=v(l_1+l_2)\ge 28$, by Lemmas  \ref{lem:ratioproperties2} (1), \ref{lem:boundl1} and Remark \ref{rem:Rvalues}, we have
		$$R(l_1+ l_2,d)=\frac{R(l_1,d)\cdot R(l_2,d)}{\tbinom{v({l_1+ l_2})}{v({l_1})}}\leq \frac{R(l_1',d)}{\tbinom{v({l_1+ l_2})}{v({l_1})}}\leq \frac{R(l_1',d)}{\tbinom{28}{v({l_1})}}\leq \frac{R(l_1',3)}{\tbinom{28}{v({l_1})}}<1.$$
		
		Case $d\ge 18$. By Lemma \ref{lem:ratioproperties}, it suffices to show $R(l,18)<1$. For positive integers $r_0>1$ and $k_0>0$,  from Lemma \ref{lem:addr0}, we conclude that $R(r_0^{k_0},18)\ge 1$ if and only if $(r_0,k_0)=(2,1)$. Thus, by Lemma \ref{lem:boundl1}, either $l_1=l_1^\prime$ of exponential type $2^1$ or $l=l_2$, which implies $R(l,18)<1$ by Lemma \ref{lem:ratioproperties2} (1) again (in fact, if $l=l'=2^1$, then $R(l,18)=\frac{R(2^1,18)\cdot R(l_2,18)}{\tbinom{n+2}{2}}\leq \frac{R(2^1,18)}{\tbinom{n+2}{2}}<1$; if $l=l_2$, then $R(l,18)\leq\prod_{1\leq i \leq m} R(r_i^{k_i},18)<1$ since $R(r_i^{k_i},18)<1$ for all $r_i>1$). This completes the proof of the theorem.
	\end{proof}
	
	By similar arguments in the proof of Theorem \ref{thm:boundnd}  and running some local computations, we obtain additional properties of Fermat-test ratios, which will be used in the subsequent sections. 
	
	\begin{lemma}\label{lem:2eleinl}
		Let $a\ge 0$ be the multiplicity of $1$ in a subdegree sequence $l$ and $d\ge 3$. Then 
		
		\begin{enumerate}
			\item [(1)] $R(l,d)<106$;
			\item [(2)] If $l$ has at least two distinct subdegrees, then $R(l,d)<60$;
			\item [(3)] If $a\ge 2$ (resp. $a\ge 1$), then $R(l,d)<3$ (resp. $R(l,d)<11$).
		\end{enumerate} 
	\end{lemma}
	To get the bounds in (1), (2), (3), it suffices to consider the case $d=3$ by Lemma \ref{lem:ratioproperties}. Note that 106 is a uniform bound for Fermat-test ratios.
	\begin{remark}\label{rem:136}
		Let $d\ge 3$, $a\ge 2$. Let $l$ be a subdegree sequence containing 1 and $a$ with $R:=R(l,d)\ge 1$. Then $a\in\{2,3,4,6\}$. Moreover, if $a=3$ (resp. $6$), then $(l,d,R)=(3^1 2^1 1^1, 3,10/9)$, $(3^1 1^1, 3,5/3)$ (resp. $(6^1 1^1, 4, 81/64)$, $(6^1 1^2, 3,4/3)$, $(6^1 2^1 1^1, 3,40/27)$, $(6^1 1^1, 3,16/3)$). 
	\end{remark}
	
	\section{Refined bounds via special monomials}\label{ss:refinedbounds}
	In this section, based on canonical bounds, we derive several refined upper bounds for the size of automorphism groups of smooth forms containing monomials involving different primitive constituents (Lemmas \ref{lem:type2bound}, \ref{lem:classify}, \ref{lem:d1d2}, \ref{lem:boundtypeII}). The motivation for such bounds is to control the size of the automorphism group $\Aut(F)$ of a smooth form $F$ when $\Aut(F)$ is not primitive and they will be frequently used in Section \ref{ss:proof}.
	
	Smoothness of a polynomial $F$ sometimes implies existence of monomials of special shape.
	\begin{lemma}\label{lem:smtosm}
		Let $F=F(x_1,\ldots,x_{k+a})$ be a smooth form of degree $d\ge 3$ with $k\ge 2$ and $a>0$. We write $F=F_1+F_2$, where  $F_1\in \C[x_1,\ldots,x_k]$ and $F_2\in (x_{k+1},\ldots,x_{k+a})\cdot \C[x_1,\ldots,x_{k+a}]$ are forms of degree $d$. Suppose $F_1$ is not smooth. Then $F_2\notin (x_{k+1},\ldots,x_{k+a})^2\cdot \C[x_1,\ldots,x_{k+a}]$. In particular, there exist integers $d_1,\ldots,d_k\ge 0$ and a monomial $\mathfrak{m}\in F$ with $\mathfrak{m}=x_1^{d_1}\cdots x_k^{d_k}\cdot x_{k+j}$ for some $1\le j\le a$.
	\end{lemma}
	\begin{proof}
		Since $F_1$ is not smooth, we may assume all partial derivatives of $F_1$ vanishes at $(1,0,\dots,0)$ $\in\C^k$.  Then by direct computation, if $F_2\in (x_{k+1},\ldots,x_{k+a})^2\cdot \C[x_1,\ldots,x_{k+a}]$, all partial derivatives of $F$ vanishes at $(1,0,\ldots,0)\in\C^{k+a}$, a contradiction. Therefore, $F_2\notin (x_{k+1},\ldots,x_{k+a})^2\cdot \C[x_1,\ldots,x_{k+a}]$, which implies the last statement in the lemma.
	\end{proof}
	
	The following lemma will be used in the proof of Lemmas \ref{lem:type2bound} and \ref{lem:classify}.
	
	\begin{lemma}\label{lem:n2inva}
		Let $F=F(x_1,\dots,x_{n+2})$ be a form of degree $d\geq3$ with $n\ge 1$ and $G$ be a finite subgroup of $\Aut(F)$. Suppose the underlying basis of $\C^{n+2}$ is compatible with a decomposition $\C^{n+2}=V_1\oplus V_2$, where $V_i$ are $G$-stable subspaces of dimension $n_i\ge 1$ ($i=1, 2$). If the $G$-space $V_2$ is irreducible and the $S^{d-1}(V_1^*)\otimes S^1(V_2^*)$-component $F^{(d-1,1)}$ of $F$ is not zero, then we have
		$$F^{(d-1,1)}=F'_1x_{n_1+1}+\cdots+F'_{n_2}x_{n_1+n_2},$$
		where $F_{i}'=F_{i}'(x_1,\dots,x_{n_1})$ are $n_2$ $\C$-linearly independent forms of degree $d-1$.
	\end{lemma}
	\begin{proof}
		Recall that $F^{(d-1,1)}\in S^{d-1}(V_1^*)\otimes S^{1}(V_2^*)$.  Let $W_2'$ be the minimal subvector space of $V_2^*$ such that $F^{(d-1,1)}\in S^{d-1}(V_1^*)\otimes S^1(W_2')$. Since $ F^{(d-1,1)}\neq 0$, we have $W_2'\neq 0$. Since $G$ preserves $F$ and $V_i$ $(i=1,2)$ are $G$-stable, we have $G$ preserves $F^{(d-1,1)}$, which implies $W_2'$ is $G$-stable. From the irreducibility of $V_2$, we get that $V_2^*$ is also irreducible under the (right) action by $G$. Then $W_2'=V_2^*$, which implies a decomposition of $F^{(d-1,1)}$ as in the lemma.
	\end{proof}
	
	For our purposes and conventions, we recall some notations in $\bf{Set\textrm{-}up}$ \ref{setup}.
	
	\begin{setup}\label{setupGF}
		Let $F=F(x_1,\dots,x_{r})$ be a smooth form of degree $d\geq3$ with $r\ge 3$. Let $G$ be a subgroup of $\Aut(F)$. Let $V_i,W_{ij},K_i, k_i,r_{ij},H_{ij},m, s, P, \widetilde{H}_{ij}, \phi_{ij}, \pi_{ij}, \phi, N$ be as in $\bf{Set\textrm{-}up}$ \ref{setup} and Definition \ref{def:associ}. Let $\tilde{\bar{H}}_{ij}:=\phi_{ij}(P)$ and $\bar{H}_{ij}:=\pi_{ij}(\tilde{\bar{H}}_{ij})$. We assume that the underlying basis of $\C^{r}$ is compatible with the decomposition $\C^{r}=\oplus_{i,j} W_{ij}$. The decomposition 
		\begin{equation*}
			S^d({\C^r}^*)=\Oplus_{\sum_{i,j}d_{ij}=d}S^{d_{11}}(W_{11}^*)\otimes S^{d_{1 2}}(W_{1 2}^*)\otimes\cdots\otimes S^{d_{m k_m}}(W_{m k_m}^*)
		\end{equation*} 
		naturally induces a decomposition 
		\begin{equation*}\label{eq:decompF}
			F=\sum_{\sum_{i,j}d_{ij}=d} F^{(d_{11},d_{12},\ldots,d_{m k_m})}\in S^d(\C^{r*}).
		\end{equation*}
		We use $F_{ij}$ to denote the component (possibly being zero) in $S^d(W_{ij}^*)$. 
	\end{setup}
	
	Note that $\tilde{\bar{H}}_{ij}$ and $\bar{H}_{ij}$ are normal subgroups of $\widetilde{H}_{ij}$ and $H_{ij}$ respectively, since $\phi_{ij}$ and $\pi_{ij}$ are surjective and $P$ is normal in $G$. Clearly primitive constituents $H_{ij}$ of $G$ preserve $F_{ij}$. 
	In the following lemmas and their proofs, we adopt the notations in $\bf{Set\textrm{-}up}$ \ref{setupGF}.
	From the associated exact sequences \eqref{eq:1associ} and \eqref{eq:2associ}, we have 
		\begin{equation}\label{eq:refinedbound}
			|G|\leq \frac{B(G,F)}{\prod_{i}[S_{k_i}:K_i]}\cdot \frac{|N|}{d^s}\cdot\frac{|P/N|}{\prod_{i,j}|H_{ij}|}.
	\end{equation}
	\begin{lemma}\label{lem:type2bound}
		Suppose $m\ge 2$ and let $V_1':=\oplus_{i\neq 2} V_i$.  If there is a monomial  $\mathfrak{m}\in S^{d-1}(V_1'^*)\otimes S^1(V_2^*)$ in $F$, then
		$$|G|\leq \frac{B(G,F)}{|H_{21}|^{k_2}}\le \frac{R(l(G),d)\cdot d^{r} \cdot r!}{{\rm JC}(r_{21})^{k_2}}.$$
	\end{lemma}
	\begin{proof}
		For all $1 \leq i \leq m$ and for all $1 \leq j \leq k_i$, we have the following commutative diagram:
		\begin{center}
			\begin{tikzcd}
				P \arrow[r, "\phi_{ij}|P",two heads] \arrow[d, hook] & \tilde{\bar{H}}_{ij} \arrow[d, hook] \arrow[r, " \pi_{ij}|\tilde{\bar{H}}_{ij}\;\;\;", two heads] & \bar{H}_{ij} \arrow[d, hook] \\
				{\rm Stab}_G(W_{ij}) \arrow[r, "\phi_{ij}",two heads]           & \widetilde{H}_{ij} \arrow[r, "\pi_{ij}", two heads]                       & H_{ij}\;.                      
			\end{tikzcd}
		\end{center}
		Note that the three injections are inclusions of normal subgroups. Consider the exact sequence 
		$$1\longrightarrow \widetilde{N}\longrightarrow P\longrightarrow H_{11}\times\cdots\times H_{1k_1}\times H_{31}\times\cdots\times H_{3k_1}\times\cdots\times H_{m1}\times\cdots\times H_{mk_m},$$
		where the last morphism is given by $g\mapsto([g|W_{11}],\ldots,[g|W_{1k_1}],[g|W_{31}],\ldots,[g|W_{mk_m}])$ and $\widetilde{N}$ is defined as its kernel. In order to prove the lemma, it suffices to show $\widetilde{N}\subset N$.
		Recall that the underlying basis $(e_1,\ldots,e_{r})$ of $\C^{r}$ is compatible with the decomposition  $\C^{r}=\oplus_{i,j} W_{ij}$. For any $$A=\Diag(\lambda_{11}I_{r_{11}},\dots,\lambda_{1k}I_{r_{1k_1}},A_{21},A_{22},\dots,A_{2k_2},\lambda_{31}I_{r_{31}},\dots,\lambda_{mk_m}I_{r_{mk_m}})\in\widetilde{N},$$
		$$B=\Diag(B_{11},\dots,B_{1k_1},\dots,B_{m1},\dots,B_{mk_m})\in P,$$
		we have the commutator $C=[A,B]$ is the identity matrix, by $C(F)=F$, the shape of $A$, $B$ and Lemma \ref{lem:n2inva}. In fact, from the decomposition $\C^r=V_1'\oplus V_2$, we have a decomposition $F=\sum_{i=0}^dF_i$, where $F_i\in S^{d-i}(V_1'^*)\otimes S^i(V_2^*)$. By the existence of the monomial  $\mathfrak{m}$, we have $F_1\neq0$. Then by Lemma \ref{lem:n2inva}, we have $F_1=\sum_{1\leq i\leq k_2, 1\leq j\leq r_{2i}}F_{ij}'x_{ij}$, where $F_{ij}'\in S^{d-1}(V_1'^*)$ are $k_2\cdot r_{21}$ $\C$-linearly independent forms of degree $d-1$ and $\{x_{ij}\}$ forms a basis of $V_2^*$. 
			Since the commutator
			$C=\Diag(I_{r_{11}},\dots,I_{r_{1k_1}},[A_{21},B_{21}],[A_{22},B_{22}],\dots,[A_{2k_2},B_{2k_2}],I_{r_{31}},\dots,I_{r_{mk_m}})$ and $C(F)=F$, we infer that $C(F_1)=F_1$ and $[A_{2i},B_{2i}]=I_{r_{2i}}$ for all $i$.
		
		Therefore, we have $\phi_{2j}(A)\in Z(\tilde{\bar{H}}_{2j} )$ for any $1\leq j\leq k_2$. Since $\tilde{\bar{H}}_{2j}\vartriangleleft\widetilde{H}_{2j}$, we have $Z(\tilde{\bar{H}}_{2j})\vartriangleleft\widetilde{H}_{2j}$. Then by primitivity of $\widetilde{H}_{2j}$, we have
		$$Z(\tilde{\bar{H}}_{2j})\subset Z(\widetilde{H}_{2j})\subset\langle\{\lambda I_{r_{2j}}|\lambda\in\C^\times\}\rangle.$$
		Thus $A_{2j}\in\langle\{\lambda I_{r_{2j}}|\lambda\in\C^\times\}\rangle$ for any $1\leq j\leq k_2$. Therefore, $A\in N$ and $N=\widetilde{N}$. Then $|P/N|=|P/\widetilde{N}|\leq \frac{\prod_{i,j}|H_{ij}|}{\prod_{1\leq j\leq k_2}|H_{2j}|}$ and the lemma follows from \eqref{eq:refinedbound}.
	\end{proof}

	By similar arguments in Lemma \ref{lem:type2bound}, we obtain the following lemma using \eqref{eq:refinedbound}.
	
	\begin{lemma}\label{lem:classify}
		Suppose $1\le a\le m$ and $k_a\ge 2$. If there is a monomial $\mathfrak{m}\in S^{d-1}(W_{a 1}^*)\otimes S^1(W_{a 2}^*)$ in $F$, then one of the following statements holds:
		\begin{enumerate}
			\item [(1)] $k_{a}=2$ and $|G|\leq\frac{B(G,F)}{|H_{a 1}|}\le \frac{R(l(G),d)\cdot d^{r} \cdot r!}{{\rm JC}(r_{a 1})}$;
			\item [(2)] $k_{a}\geq3$ and $|G|\leq\frac{B(G,F)}{2\cdot |H_{a 1}|}\le \frac{R(l(G),d)\cdot d^{r} \cdot r!}{2\cdot {\rm JC}(r_{a 1})}$.
		\end{enumerate}
	\end{lemma}
	Note that the spaces $\oplus_{(i,j)\neq(a,2)} W_{ij}$ and $W_{a2}$ are ${\rm Stab}_G(W_{a2})$-stable subspaces of $\C^r$. Since $W_{a2}$ is an irreducible ${\rm Stab}_G(W_{a2})$-stable subspace, we can apply Lemma \ref{lem:n2inva} as in the proof of Lemma \ref{lem:type2bound} to prove that $N$ is equal to the kernel of the morphism $P\longrightarrow \prod_{(i,j)\neq(a,2)}H_{ij}$, $g\mapsto([g|W_{ij}])$. From this, we have $|P/N|\leq \frac{\prod_{i,j}|H_{ij}|}{|H_{a2}|}$ in Lemma \ref{lem:classify} (1) and (2). For $k_a\geq 3$, 
		if $[S_{k_a}:K_a]\geq2$, then (2) follows; if $[S_{k_a}:K_a]=1$, then $\frac{|N|}{d^s}\leq \frac{1}{d^{k_a-1}}$ (similar to the proof of Lemma \ref{lem:boundtypeII} (2)) and we are done by \eqref{eq:refinedbound}.
	\begin{lemma}\label{lem:d1d2}
		Let $1\le i_1\le m$, $1\le j_1\le k_{i_1}$, $0<d_1<d$. Suppose there is a monomial $\mathfrak{m}\in S^{d_1}(W_{i_1 j_1}^*)\otimes S^{d-d_1}(\oplus_{(i_2,j_2)\neq (i_1,j_1)} W_{i_2 j_2}^*)$ in $F$. Then $\bar{H}_{i_1 j_1}$ contains a normal subgroup $\bar{\bar{H}}_{i_1 j_1}$ preserving a non-zero form of degree $d_1$ and $$|G|\le \frac{B(G,F)}{[H_{i_1 j_1}:\bar{\bar{H}}_{i_1 j_1}]}\le \frac{R(l(G),d)\cdot|\bar{\bar{H}}_{i_1 j_1}|\cdot d^r\cdot r!}{{\rm JC}(r_{i_1 j_1})}.$$
	\end{lemma}
	
	\begin{proof}
		Without loss of generality, we assume $(i_1,j_1)=(1,1)$. Consider the exact sequence 
		\begin{equation}\label{eq:N11}
			1\longrightarrow \widetilde{N}_{11}\longrightarrow P\longrightarrow H_{12}\times\cdots\times H_{1k_1}\times H_{21}\times\cdots\times H_{2k_2}\times\cdots\times H_{m1}\times\cdots\times H_{mk_m},
		\end{equation}
		where the last morphism is given by $g\mapsto([g|W_{12}],\ldots,[g|W_{1k_1}],[g|W_{21}],\ldots,[g|W_{mk_m}])$ and $\widetilde{N}_{11}$ is defined as its kernel. Let $\bar{\bar{H}}_{1 1}:=\pi_{11}\circ\phi_{11}(\widetilde{N}_{11})$. Then $\bar{\bar{H}}_{1 1}\vartriangleleft \bar{H}_{11}$ and we have the exact sequence
		\begin{equation}\label{eq:barbarH}
			1\longrightarrow N\longrightarrow \widetilde{N}_{11}\longrightarrow \bar{\bar{H}}_{11}\longrightarrow 1.
		\end{equation}
		From the exact sequences \eqref{eq:1associ}, \eqref{eq:2associ}, \eqref{eq:N11}, \eqref{eq:barbarH} and by Theorem \ref{thm:diag}, we have $$|G|\leq |P|\prod_{ i} k_i !\leq \frac{|\widetilde{N}_{11}| \prod_{i} k_i! \prod_{i,j} |H_{ij}|}{|H_{11}|}\leq \frac{|N| |\bar{\bar{H}}_{11}| \prod_{i} k_i! \prod_{i,j} |H_{ij}|}{|H_{11}|}\leq  \frac{d^s \prod_{i} k_i! \prod_{i,j} |H_{ij}|}{[H_{11}:\bar{\bar{H}}_{11}]}, $$
		which implies the inequalities in the lemma by definitions of canonical bounds and Fermat-test ratios. From the existence of $\mathfrak{m}$, we infer that $\bar{\bar{H}}_{1 1}$ preserves a non-zero form of degree $d_1$.
	\end{proof}
	
	\begin{lemma}\label{lem:boundtypeII}
		Let $1\le i_1\le m$, $1\le j_1\le k_{i_1}$ and $r_{i_1 j_1}\ge 2$. We define $$\mathcal{A}:=\{(i_2,j_2)|\; (i_2,j_2)\neq (i_1,j_1),\; \exists \text{ a monomial }\mathfrak{m}\in S^{d-1}(W_{i_1 j_1}^*)\otimes S^{1}(W_{i_2 j_2}^*) \text{ in }F \}.$$ Let $c$ be the cardinality of $\mathcal{A}$. If the form $F_{i_1 j_1}$ is either zero or not smooth, then $c\neq 0$ and the following statements hold:
		\begin{enumerate}
			\item [(1)]  $|G|\le\frac{B(G,F)}{|H_{i_2 j_2}|}\le \frac{R(l(G),d)\cdot d^{r} \cdot r!}{{\rm JC}(r_{i_2 j_2})}$ for all $(i_2,j_2)\in \mathcal{A}$;
			
			\item [(2)] $|G|\le\frac{B(G,F)}{d^{c-1}} \le \frac{R(l(G),d)\cdot d^{r} \cdot r!}{d^{c-1}}$;
			
			\item [(3)] If $c=1$, $r_{i_2 j_2}=1$ for $(i_2,j_2)\in\mathcal{A}$ and $G$ contains $\mathcal{I}_{r,d}$, then $F_{i_1 j_1}$ is not zero and $H_{i_1 j_1}$ admits an $F_{i_1 j_1}$-lifting preserving some non-zero form of degree $d-1$. 
		\end{enumerate}
	\end{lemma}
	
	\begin{proof}
		We may assume $(i_1,j_1)=(1,1)$. We choose a basis of $\C^r$ compatible with the decomposition $\C^r=\oplus_{i,j}W_{ij}$. By the smoothness of $F$ and Lemma \ref{lem:smtosm}, we have $c\neq 0$. The statement (1) is a direct consequence of Lemmas \ref{lem:type2bound} and \ref{lem:classify}.
		
		For (2),  let $A={\rm diag}(\lambda_{11} I_{r_{11}},\ldots,\lambda_{1k_1}I_{r_{1k_1}},\lambda_{21}I_{r_{21}},\ldots,\lambda_{m1}I_{r_{m 1}},\ldots,\lambda_{mk_m}I_{r_{m k_m}})\in N$. By $A(F)=F$, we get $\lambda_{11}^{d-1}\lambda_{i_2 j_2}=\cdots =\lambda_{11}^{d-1}\lambda_{i_{c+1} j_{c+1}}=1$, where $\mathcal{A}=\{(i_2,j_2),\ldots,(i_{c+1},j_{c+1})\}$. Thus, $\lambda_{i_2 j_2}=\cdots=\lambda_{i_{c+1} j_{c+1}}$, which implies $|N|\le d^{s-(c-1)}$ by Theorem \ref{thm:diag}. From this, we obtain the inequalities in (2) using the associated exact sequences \eqref{eq:1associ} and \eqref{eq:2associ}.
		
		For (3), we define $W':=\oplus_{(i',j')\neq (1,1), (i_2,j_2)} W_{i'j'}$, where $\mathcal{A}=\{(i_2,j_2)\}$. Let $a:=r_{i_2 j_2}=1$ and $b:={\rm dim}(W')$. Since $r=r_{11}+a+b\ge (a+1)+a+b$, we have $2a+b\le r-1$. If the $S^d(W_{11}^*)$-component $F_{11}$ of $F$ is zero, then we have $$F\in \sum_{1\le d_1\le d}S^{d_1}(W_{i_2 j_2}^*)\otimes S^{d-d_1}(W_{11}^*\oplus W'^*)+ \sum_{2\le d_2\le d}S^{d_2}(W'^*)\otimes S^{d-d_2}(W_{11}^*\oplus W_{i_2 j_2}^*),$$
		which contradicts the smoothness of $F$ by \cite[Lemma 3.2]{OY19}. Therefore, $F_{11}\neq 0$. Since the $S^{d-1}(W_{11}^*)\otimes S^1(W_{i_2 j_2}^*)$-component $F_{11i_2j_2}$ of $F$ is not zero, we write  $F_{11i_2j_2}=F_{11}'\cdot L_{i_2 j_2}'$, where $F_{11}'\in S^{d-1}(W_{11}^*)$ and $L_{i_2 j_2}'\in S^1(W_{i_2 j_2}^*)$.  By $\mathcal{A}=\{(i_2,j_2)\}$, we have $W_{11}\oplus W_{i_2j_2}$ is ${\rm Stab}_G(W_{11})$-stable. Let $M:=\{(g_1,g_2)|\; g_1=g|W_{11}, g_2=g|W_{i_2 j_2}, g\in {\rm Stab}_G(W_{11})\}$. If $A=(\lambda_1 I_{r_{11}},\lambda_2)\in M$, then by $A(F)=F$, we have $\lambda_1^d=\lambda_1^{d-1}\lambda_2=1$ and $\lambda_1=\lambda_2$. Consider the surjective homomorphisms $\varphi_1: M\longrightarrow H_{11}$, $(g_1,g_2)\mapsto [g_1]$ and $\varphi_2: M\longrightarrow \widetilde{H}_{i_2 j_2}\cong C_d$, $(g_1,g_2)\mapsto g_2$. For any $(g_1, 1)\in {\rm Ker}(\varphi_2)$, we have $g_1$ preserves both $F_{11}$ and $F_{11i_2 j_2}$, which implies that $\varphi_1| {\rm Ker}(\varphi_2): {\rm Ker}(\varphi_2)\rightarrow H_{11}$ is injective. Clearly $|H_{11}|=|M|/d=|{\rm Ker}(\varphi_2)|$. Thus $\{ g_1 |\; (g_1, 1)\in {\rm Ker}(\varphi_2)\}\cong H_{11}$ is a desired $F_{11}$-lifting of $H_{11}$. This completes the proof of the lemma. 
	\end{proof}
	
	\section{Primitive groups of automorphisms}\label{ss:prim}
	Finite linear groups in small dimensions have been completely classified (see e.g. \cite{Fei70} for a survey).
	In particular, all finite primitive subgroups in $\PGL(r,\C)$ for small  $r$ are known. Based on this, we derive the following theorem which is the main result of this section.
	
	\begin{theorem}\label{thm:prim}
		Fix integers $n\geq1$, $d\geq3$. Let $F=F(x_1,\ldots,x_{n+2})$ be a smooth form of degree $d\ge 3$ with primitive $\Aut(F)$ and $|\Aut(F)|\ge d^{n+2}\cdot (n+2)!$. Then up to isomorphism, $F$ is as follows:
		\begin{footnotesize}
			\begin{table}[H]\rm
				\renewcommand\arraystretch{1.2}
				\begin{tabular}{p{0.3in}p{0.7in}p{0.6in}p{4.0in}}
					\hline
					$(n,d)$&${\rm Lin}(X_F)$&$|{\rm Lin}(X_F)|$&$F$\\
					\hline
					$(1,4)$&$\PSL(2,7)$&$168$&$x_1^3x_2+x_2^3x_3+x_3^3x_1$ \\
					$(1,6)$&$ C_3^2\rtimes \SL(2,3)$&$216$&$x_1^6 +x_2^6 + x_3^6-10(x_1^3x_2^3 + x_2^3x_3^3 + x_3^3x_1^3)$ \\
					%                \hline
					$(1,6)$&$A_6$&$360$&$10x_1^3x_2^3+9(x_1^5+x_2^5)x_3-45x_1^2x_2^2x_3^2-135x_1x_2x_3^4+27x_3^6$ \\
					$(2,4)$&$C_2^4.S_5$&$1920$&$x_1^4 + x_2^4 + x_3^4 + x_4^4 + 12x_1x_2x_3x_4$ \\
					$(4,6)$&${\rm PSU}(4,3).C_2$&$6531840$&$\sum\limits_{1\leq i\leq6}x_i^6+\sum\limits_{1\leq i\neq j\leq6}15x_i^4x_j^2-\sum\limits_{1\leq i<j<k\leq6}30x_i^2x_j^2x_k^2+240\sqrt{-3}x_1x_2x_3x_4x_5x_6$ \\
					\hline
				\end{tabular}
			\end{table}
		\end{footnotesize}
	\end{theorem}
	
	Besides the values of ${\rm JC}(r)$, our proof of Theorem \ref{thm:prim} relies on explicit list of the finite primitive subgroups in $\PGL(r,\C)$ for $r\in\{2,3,4,5,6,8\}$, 
	see Blichfeldt \cite{Bli17} for $r\le 4$ (see also \cite[Appendix A]{CS19} for diagrams indicating inclusions between groups in the case $r=4$); Brauer \cite[\S 9]{Bra67} for $r=5$; Lindsey \cite[\S 3]{Lin71} for $r=6$; Feit \cite[\S 2]{Fei76} for $r=8$. For ease of reference, we present a list of finite primitive groups $H$ in $\PGL(r,\C)$ with $|H|\ge 3^{r-1}r!$ and $r\in\{2,3,4,5,6\}$ (see Table \ref{tab:listofprim}).
	\begin{footnotesize}
		\begin{table}[H]\rm
			\renewcommand\arraystretch{1.1}
			\begin{tabular}{ccc|ccc|ccc|ccc}
				%			\hline
				$r$&$|H|$&$H$&$r$&$|H|$&$H$&$r$&$|H|$&$H$&$r$&$|H|$&$H$\\
				\hline
				2&12&$A_4$&
				3&360&$A_6$&
				4&1920&$C_2^4.S_5$&
				
				4&25920&$ \rm PSP(4,3)$\\
				2&24&$S_4$&
				4&720&$A_4\times A_5$&
				4&1920&$C_2^4.S_5 $&
				
				5&25920&$\rm PSP(4,3)$\\
				2&60&$A_5$&
				4&720&$S_6$&
				4&2520&$A_7$&
				6&604800&$\rm HaJ$
				\\
				3&60&$A_5$&
				4&960&$C_2^4\rtimes A_5 $&
				4&3600&$A_5^2$&
				6&3265920&$\rm PSU(4,3) $\\
				3&72&$C_3^2\rtimes Q_8$&
				4&960&$C_2^4\rtimes A_5 $&
				4&5760&$C_2^4\rtimes A_6 $&
				
				6&6531840&${\rm PSU}(4,3).C_2$\\
				3&168&$\PSL(2,7)$&
				4&1152&$S_4^2\rtimes C_2$&
				4&7200&$A_5^2\rtimes C_2 $&
				
				&&\\
				3&216&$C_3^2\rtimes \SL(2,3)$&
				4&1440&$S_4\times A_5$&
				4&11520&$C_2^4.S_6$&
				&&\\
				%			\hline
			\end{tabular}
			\caption{Large finite primitive groups $H\subset\PGL(r,\C)$}\label{tab:listofprim}
		\end{table}
	\end{footnotesize}

	Before presenting the proof of the theorem, we first give two lemmas.
	
	\begin{lemma}\label{lem:primlifting}
		Let $\widetilde{H}\subset \GL(r,\C)$ be a finite primitive subgroup and $H:=\pi(\widetilde{H})$. If $\widetilde{H}\cong H$, then the following statements hold:
		\begin{enumerate}
			\item[(1)] $H$ admits a faithful linear representation of degree $r$;
			
			\item[(2)] $H$ has no non-trivial solvable normal subgroups;
			
			\item[(3)] For each prime $p$, $H$ has no subgroups isomorphic to $C_p^r$.
		\end{enumerate}
		
	\end{lemma}
	
	\begin{proof}
		(1) is clear. By $\widetilde{H}\cong \pi(\widetilde{H})$, we have $|Z(\widetilde{H})|=1$. For (2), we suppose $N$ is a non-trivial solvable normal subgroup of $\widetilde{H}$.  Since $N$ is solvable, $N$ has a non-trivial characteristic abelian subgroup, say $N_1$. Then $N_1$ is a normal abelian subgroup of $\widetilde{H}$, which implies $N_1\subset Z(\widetilde{H})$ by primitivity of $\widetilde{H}$, a contradiction. Consequently, the statement (2) is true. For (3), we suppose that $\widetilde{H}$ has a subgroup, say $H_1$, isomorphic to $C_p^r$ for some prime $p$. Since $H_1\subset \GL(r,\C)$, we infer that $A:={\rm diag}(\xi_p,\ldots,\xi_p)\in H_1$. Thus, $A\in Z(\widetilde{H})$, a contradiction. Hence (3) is true.
	\end{proof}
	
	By a theorem of Schur, for a finite group $G$, there is a finite group $SC(G)$ such that $SC(G)$ is a central extension of the Schur multiplier $H^2(G,\C^\times)$ by $G$ and every projective representation of $G$ can be lifted to $SC(G)$ (see \cite[Chapter 7]{Rot95}). We call $SC(G)$ a {\it Schur cover group} of $G$. If $G$ is perfect (i.e., $G$ is equal to its commutator subgroup), then $SC(G)$ is perfect and unique up to isomorphism (see \cite[Corollary 11.12]{Rot95}).
	
	\begin{lemma}\label{lem:perfect}
		Let $H\subset \PGL(r,\C)$ be a finite perfect subgroup preserving a form $F=F(x_1,\ldots,x_r)$ of degree $d\ge 3$. Then the Schur cover group $SC(H)$ of $H$ admits a linear representation $\rho: SC(H)\rightarrow \GL(r,\C)$ such that $\rho (SC(H))\subset \Aut(F)$ and $\pi (\rho(SC(H)))=H$. Moreover, if $(r,d)=(4,6)$ and $F$ is smooth, then $\rho(SC(H))$ has no elements $A$ of order $5$ with ${\rm tr}(A)=-1$.
	\end{lemma}
	
	\begin{proof}
		The group $SC(H)$ has a representation $\rho: SC(H)\rightarrow \GL(r,\C)$ such that $\pi (\rho(SC(H)))=H$. Since $H$ preserves $F$, we get a character $\chi:SC(H)\rightarrow \C^\times$ satisfying $\rho(g)(F)=\chi(g)F$ for all $g\in SC(H)$. On the other hand, $SC(H)$ is perfect since $H$ is perfect, which implies $\chi$ is trivial. Thus, $\rho (SC(H))$ preserves $F$. For the last statement in the lemma, we suppose $A={\rm diag}(\xi_5,\xi_5^2,\xi_5^3,\xi_5^4)\in \Aut(F)$. Then by $A(F)=F$, we have $x_1^5x_i\notin F$ for any $1\le i\le 4$, which contradicts smoothness of $F$.
	\end{proof}
	We are now ready to prove the main result of this section.
	\begin{proof}[Proof of Theorem \ref{thm:prim}] 
		The  primitivity of $\Aut(F)$ implies the only primitive constituent $H$ of $\Aut(F)$ is $\pi(\Aut(F))\cong \Aut(F)/Z(\Aut(F))$, where $Z(\Aut(F))\cong C_d$. Hence, we have
		$|\Aut(F)|=B(\Aut(F))=d\cdot |H|$ and $|H|\geq d^{n+1}(n+2)!$. Since $|H|\leq{\rm JC}(n+2)$, it suffices to consider the cases where $(n,d)$ is one of the following: $(1,3)$, $(1,4)$, $(1,5)$, $(1,6)$, $(1,7)$, $(2,3)$, $(2,4)$, $(2,5)$, $(2,6)$, $(2,7)$, $(2,8)$, $(2,9)$, $(2,10)$, $(3,3)$, $(4,3)$, $(4,4)$, $(4,5)$, $(4,6)$, $(6,3)$.

		For $(n,d)$ equal to $(1,4)$, $(1,5)$, $(1,6)$, and $(1,7)$, the results are already known, see \cite{Pam13} and \cite{Haru19}; for $(n,d)=(1,6)$ and $|{\rm Lin}(X_F)|=216$, see also \cite{BB22}; for $(n,d)=(1,3)$, see e.g. \cite[Lemma 3.12]{YYZ23}. For $2\le n\le 4$, Fermat cubic $n$-fold has the largest possible order for the automorphism group among all smooth cubic $n$-folds (see \cite{Seg42}, \cite{Hos97}, \cite{Do12} for $n=2$; \cite{WY20} for $n=3$; \cite{LZ22}, \cite{YYZ23} for $n=4$). For the case $(n,d)=(2,4)$, $F$ with primitive ${\rm Lin}(X_F)$ of order $\ge 4^3\cdot 4!=1536$ is unique (up to isomorphism) and $|{\rm Lin}(X_F)|=1920$ (see \cite[Exercise 6, Chap. XVII]{Bur55}, \cite[Theorem A and Corollary B]{AOT24}).
		
		Cases $(n,d)=(2,5), (2,7), (2,9), (4,5), (6,3)$. Since $n+2$ and $d$ are coprime, by \cite[Theorem 3.5]{GLM23}, there is a (primitive) subgroup $\widetilde{H}\subset \Aut(F)$ with $\widetilde{H}\cong \pi(\widetilde{H})=H$. Thus, $H$ satisfies Lemma \ref{lem:primlifting} (1)-(3). However, no such $H$ exists by Table \ref{tab:listofprim} and \cite[Theorem A]{Fei76}.
		
		For $(n,d)=(2,6)$, we may assume $|H|\geq 6^3\cdot4!=5184$. We then need to consider $H$ to be one of the following groups: $A_5^2\rtimes C_2$, $C_2^4\rtimes A_6$, $C_2^4.S_6$, ${\rm PSP}(4,3).$ Suppose $H\cong A_5^2\rtimes C_2$. Then $H$ has a perfect subgroup $H_1\cong A_5^2$. The Schur cover group of $A_5$ (resp. $A_5^2$) is $\SL(2,5)$ (resp. $\SL(2,5)^2$) of order $120$ (resp. $14400$). From the character table of $\SL(2,5)$ (see \cite[Page 2]{CCNPW}), we infer that for each irreducible representation $\rho: \SL(2,5)^2\rightarrow \GL(4,\C)$, there is $A\in\rho(\SL(2,5)^2)$ with $A$ of order $5$ and ${\rm tr}(A)=-1$, which contradicts Lemma \ref{lem:perfect}. Thus, $H\ncong A_5^2\rtimes C_2$. Similarly, we have $H\ncong C_2^4\rtimes A_6$, $C_2^4.S_6$, ${\rm PSP}(4,3)$ by choosing $H_1$ as $A_6$, $A_6$, ${\rm PSP}(4,3)$ respectively.
		
		For $(n,d)=(2,8),(2,10)$, we may assume $|H|\geq 8^3\cdot4!=12288$. Thus we only need to consider $H\cong {\rm PSP}(4,3)$, which preserves no forms of degree $d=8,10$ by the Molien's series of its Schur cover (see \cite[Page 287]{ST54}), a contradiction by Lemma \ref{lem:perfect}.
		
		For $(n,d)=(4,4)$,  since $|H|\geq 4^5\cdot6!=737280$, we have $H\cong {\rm PSU}(4,3)$ or $H\cong {\rm PSU}(4,3).C_2$. As in the previous cases, these two groups can be ruled out by the Molien's series of the Schur cover of ${\rm PSU}(4,3)$ (see \cite[Page 86]{Tod50}) and the fact ${\rm PSU}(4,3)<{\rm PSU}(4,3).C_2$.
		
		For $(n,d)=(4,6)$, there is a unique form $F=F(x_1,\ldots,x_6)$ of degree $6$ with primitive $\pi (\Aut(F))\cong {\rm PSU}(4,3).C_2$ of order $6531840$ (\cite[\S 3, \S 4]{Tod50}). Smoothness of $F$ can be verified by Magma (\cite{BCP97}).
	\end{proof}
	
	Forms preserved by finite primitive groups in $\PGL(2,\C)$ are known (see \cite[\S 105]{MBD16}).
	
	\begin{lemma}\label{lem:A4S4A5}
		Let $H\subset \PGL(2,\C)$ be a finite primitive group preserving a form $F=F(x_1,x_2)$ of degree $d\ge 1$. If $H= A_4$ (resp. $S_4$, resp. $A_5$), then $d\ge 4$ (resp. $d\ge 6$, resp. $d=12$ or $d\ge 20$). Moreover, if $(H,d)=(A_4,4)$ (resp. $(S_4,6)$, resp. $(A_5, 12)$), then $F$ is isomorphic to the smooth form $x_1^4+2\sqrt{-3}x_1^2x_2^2+x_2^4$ (resp. $x_1^5x_2+x_2^5x_1$, resp. $x_1^{11}x_2+11x_1^6x_2^6-x_1x_2^{11}$).
	\end{lemma}
	
	\section{Proof of the main theorem}\label{ss:proof}
	 In this section, using Theorem \ref{thm:prim} and the refined bounds in Section \ref{ss:refinedbounds}, we classify smooth forms with large imprimitive or reducible automorphism groups in \S \ref{ss:7.1} and \S \ref{ss:7.2} (Theorems \ref{thm:imprim} and \ref{thm:reducible}). In \S \ref{ss:7.3}, combining these results, we prove a stronger form (Theorem \ref{thm:atleastFermat}) of Theorem \ref{thm:Main}.
	
	For a smooth form $F$ of degree $d\ge 3$ in $r\ge 1$ variables, we define $\mathfrak{R}_F:=|\Aut(F)|/(d^r\cdot r!)$.
	
	\subsection{Imprimitive $\Aut(F)$}\label{ss:7.1} Recall that by an imprimitive finite linear group we mean that it is irreducible and not primitive. Now we classify smooth forms with large imprimitive automorphism groups.
	\begin{theorem}\label{thm:imprim}
		Let $F=F(x_1,\ldots,x_{n+2})$ be a smooth form of degree $d$, where $n\geq 1$, $d\geq3$. Suppose $\Aut(F)$ is imprimitive and $F$ is not isomorphic to Fermat form $F_d^n$. If $|\Aut(F)|\ge d^{n+2}\cdot (n+2)!$, then up to isomorphism, $F$ is as follows:
		\begin{footnotesize}
			\begin{table}[H]\rm
				\renewcommand\arraystretch{1.2}
				\begin{tabular}{p{0.3in}p{0.76in}p{0.6in}p{4.0in}}
					\hline
					$(n,d)$&$\Aut(F)$&$|\Aut(F)|$&$F$\\
					\hline
					$(2,6)$&$ C_{6}^2.(S_4^2\rtimes C_2)$&$41472$&$x_1^5x_2+x_2^5x_1+x_3^5x_4+x_4^5x_3$ \\
					%                \hline
					$(2,12)$&$ C_{12}^2.(A_5^2\rtimes C_2) $&$ 1036800$&$x_1^{11}x_2+11x_1^6x_2^6-x_1x_2^{11}+x_3^{11}x_4+11x_3^6x_4^6-x_3x_4^{11}$ \\
					%                \hline
					$(4,12)$&$ C_{12}^3.(A_5^3\rtimes S_3) $&$2239488000$&$x_1^{11}x_2+11x_1^6x_2^6-x_1x_2^{11}+x_3^{11}x_4+11x_3^6x_4^6-x_3x_4^{11}+x_5^{11}x_6+11x_5^6x_6^6-x_5x_6^{11}$ \\
					\hline
				\end{tabular}
			\end{table}
		\end{footnotesize}
	\end{theorem}
	\begin{proof}
		Since $\Aut(F)$ is imprimitive and $F$ is not isomorphic to the Fermat form, $l(\Aut(F))$ is of exponential type $r^k$ with $r\ge 2$ and $k\ge 2$, and $\Aut(F)$ has a primitive-constituent sequence $\mathcal{H}:=(H_1,\ldots,H_k)$ with $H_i\subset\PGL(r,\C)$ and $H_i\cong H_j$ for $1\le i,j\le k$. By Lemma \ref{lem:FTRandGF}, we have $R(\mathcal{H},d)\ge 1$. By direct calculation (see (\ref{eq:decreasing})), we get $|H_i|> d^{r-1}\cdot r!$. In fact, $R(\mathcal{H},d)=\frac{d^k\cdot |H_i|^k\cdot k!}{d^{rk}\cdot (rk)!}\ge 1$ implies $|H_i|>d^{r-1}\cdot r!$ since $(rk)!> (r!)^k\cdot k!$. It suffices to consider two cases: $r\ge 3$ and $r=2$.
		
		Case $r\ge 3$. If $H_1$ preserves a smooth form $F_1(x_1,\ldots,x_r)$ of degree $d$, then by Theorem \ref{thm:prim}, we have $(r,d,|H_1|)=(3,4,168)$, $(3,6,360)$, $(4,4,1920)$, $(6,6,6531840)$. By direct calculation, we have $R(\mathcal{H},d)<1$, a contradiction. Thus, $H_1$ preserves no smooth forms of degree $d$. Since ${\rm JC}(r)>106$ for $r\ge 3$, by Lemmas \ref{lem:smtosm}, \ref{lem:classify} and \ref{lem:2eleinl}, we get $|\Aut(F)|<d^{n+2}\cdot (n+2)!$, a contradiction. In fact, we may write $F=F_1+F_2$, where  $F_1\in \C[x_1,\ldots,x_r]$ and $F_2\in (x_{r+1},\ldots,x_{n+2})\cdot \C[x_1,\ldots,x_{n+2}]$ are forms of degree $d$ with $H_1$ preserving $F_1$. Then $F_1$ is not smooth and by Lemma \ref{lem:smtosm}, there exist integers $d_1,\ldots,d_r\ge 0$ and a monomial $\mathfrak{m}\in F$ with $\mathfrak{m}=x_1^{d_1}\cdots x_r^{d_r}\cdot x_{r+j}$ for some $1\le j\le n+2-r$. Thus, by Lemmas \ref{lem:classify} and \ref{lem:2eleinl} (1), we have $$|\Aut(F)|\le \frac{R(l(\Aut(F)),d)\cdot d^{n+2} \cdot (n+2)!}{{\rm JC}(r)}< \frac{106\cdot d^{n+2} \cdot (n+2)!}{{\rm JC}(r)}<d^{n+2}\cdot (n+2)!.$$
		
		Case $r=2$. If $H_1$ preserves no smooth forms of degree $d$, then $|\Aut(F)|<d^{n+2}\cdot (n+2)!$ by similar arguments as in case $r\ge 3$. Thus, we may assume $H_1$ preserves a smooth form $F_1(x_1,\ldots,x_r)$ of degree $d$. By Lemma \ref{lem:A4S4A5}, if $H_1\cong A_4, S_4, A_5$, then $H_1$ preserves no smooth forms in $2$ variables of degree less than $4,6,12$ respectively. If $H_1\cong A_4$, then $d\ge 4$ and $R(\mathcal{H},d)< 1$, a contradiction. If $H_1\cong S_4$, then $d\ge 6$ and $R(\mathcal{H},d)\ge 1$ implies $(k,d)=(2,6)$. Then $R(\mathcal{H},d)=4/3$ and $F$ can be expressed as $F_1(x_1,x_2)+F_2(x_3,x_4)+F'$, where $F_i$ are preserved by $H_i$ and $F'$ is in the intersection of the two ideals  $(x_1,x_2)\cdot\C[x_1,x_2,x_3,x_4]$ and $(x_3,x_4)\cdot\C[x_1,x_2,x_3,x_4]$.
		If $F'\neq 0$, then by Lemma \ref{lem:d1d2}, $H_1$ has a normal subgroup $\bar{H}_1$ and $|\Aut(F)|\le \frac{B(\Aut(F),F)}{[H_1:\bar{\bar{H}}_1]}$, where $\bar{\bar{H}}_1$ is a normal subgroup of $\bar{H}_1$ and $\bar{\bar{H}}_1$ preserves a non-zero form of degree $j$ for some $j\in\{1,2,3,4,5\}$. Then    $$1\le \frac{|\Aut(F)|}{d^{n+2}\cdot (n+2)!}\le \frac{B(\Aut(F),F)}{d^{n+2}\cdot (n+2)!\cdot [H_1:\bar{\bar{H}}_1]}=\frac{R(\mathcal{H},d)}{[H_1:\bar{\bar{H}}_1]}=\frac{4}{3\cdot [H_1:\bar{\bar{H}}_1]}$$ and $\bar{\bar{H}}_1=H_1\cong S_4$,  which contradicts Lemma \ref{lem:A4S4A5}. Thus, $F'=0$. Therefore, $F$ is unique up to isomorphism by Lemma \ref{lem:A4S4A5}. If $H_1\cong A_5$, then $d=12$ or $d\ge 20$ by Lemma \ref{lem:A4S4A5}. So $R(\mathcal{H},d)\ge 1$ implies $(k,d)=(2,12), (3,12)$. Similar to the case $H_1=S_4$, we have the uniqueness of $F$.
	\end{proof}
	
	As a consequence of Theorems \ref{thm:prim}, \ref{thm:imprim} and Lemma \ref{lem:A4S4A5}, we have the following result.
	\begin{proposition}\label{pp:irrbounds}
		Let $F=F(x_1,\ldots,x_r)$ be a smooth form of degree $d\ge 3$, where $r\ge 2$. Suppose $\Aut(F)$ is irreducible and $\mathfrak{R}_F=\frac{|\Aut(F)|}{d^r\cdot r!}>1$. Then either $\mathfrak{R}_F=5/2$ or $\mathfrak{R}_F\le 25/12$. Moreover, $\mathfrak{R}_F=5/2$ if and only if $(r,d)=(2,12)$ and $\Aut(F)/Z(\Aut(F))\cong A_5$. 
	\end{proposition}
	
	\subsection{Reducible $\Aut(F)$}\label{ss:7.2} The following lemma will be used to prove Theorem \ref{thm:reducible}.
	
	\begin{lemma}\label{lem:ratioprod}
		Let $n_1,\ldots,n_m$ be positive integers, where $m\ge 2$. Suppose $q_i$ $(1\le i\le m)$ are rational numbers satisfying $q_i=5/2$ (resp. $q_i=1$) if $n_i\ge 2$ (resp. $n_i=1$). If $\frac{\prod_{i=1}^{m}q_i\cdot n_i !}{(n_1+\cdots +n_m)!}\ge 1$, then $m=2$ and $n_1=n_2=2$.
	\end{lemma}
	
	\begin{proof}
		We may assume $n_1\ge n_2\ge \cdots \ge n_m$. 
		We write $$q:=\frac{\prod_{i=1}^{m}q_i \cdot n_i !}{(n_1+\cdots +n_m)!}=\frac{q_1 \cdot n_1! n_2!}{(n_1+n_2)!}\cdot \frac{q_2\cdot (n_1+n_2)! n_3 !}{(n_1+n_2+n_3)!}\cdots\frac{q_{m-1}q_m (n_1+\cdots +n_{m-1})! n_m!}{(n_1+\cdots +n_m)!}.$$
		By the formula above, if $n_m=1$ or $m\ge 3$ or $n_1>2$, then all fractions on the right side are less than $1$, a contradiction to $q\ge 1$. Thus, we have $m=2$, $n_1=n_2=2$ and $q=\frac{25}{24}$.
	\end{proof}
	
	Next we show that smooth homogeneous polynomials of reducible automorphism groups have fewer automorphisms than Fermat polynomials.
	\begin{theorem}\label{thm:reducible}
		Let $F=F(x_1,\ldots,x_{r})$ be a smooth form of degree $d\ge 3$, where $r\geq 3$. If $\Aut(F)$ is reducible, then $|\Aut(F)|<d^{r}\cdot r!.$
	\end{theorem}
	
	\begin{proof}
		To prove the theorem by contradiction, we suppose $\mathfrak{R}_F=\frac{|\Aut(F)|}{d^{r}\cdot r!}\ge 1$. Let $G:=\Aut(F)$. By Theorem \ref{thm:mono}, we may assume at least one subdegree of $G$ is larger than $1$. Let $V_i,W_{ij},k_i,r_{ij},H_{ij},m, F_{ij}$ be as in $\bf{Set\textrm{-}up}$ \ref{setupGF}. Recall that $\C^r=V_1\oplus \cdots \oplus V_m$,  $V_i=W_{i1}\oplus \cdots \oplus W_{ik_i}$, $H_{ij}\subset \PGL(W_{ij})$, and ${\rm dim}(W_{ij})=r_{ij}$. We define $\beta$ to be $1$ (resp. $2$) if $l(G)$ has at least two different subdegrees (resp. all subdegrees of $G$ are the same).  From the definitions, we have \begin{equation}\label{eq:dividingH}
			\frac{B(G,F)}{|H_{ij}|^{k}}\le \frac{R(l(G),d)\cdot d^r\cdot r!}{\beta\cdot {\rm JC}(r_{ij})^{k}}
		\end{equation}	
		for any pair $(i,j)$ and any $1\le k\le k_i$. By Lemmas \ref{lem:FTRandGF}  and \ref{lem:2eleinl}, we obtain $R(l(G),d)<60\beta$. If there is a monomial $$\mathfrak{m}\in S^{1}(W_{i_1 j_1}^*)\otimes S^{d-1}(\oplus_{r_{i_2 j_2}=1} W_{i_2 j_2}^*)$$ in $F$ with $r_{i_1j_1}\ge 2$, then by Lemma \ref{lem:type2bound}, we have 
		$$|\Aut(F)|\le \frac{B(G,F)}{|H_{i_1j_1}|^{k_{i_1}}}\le \frac{R(l(G),d)\cdot d^r\cdot r!}{\beta\cdot {\rm JC}(r_{i_1j_1})^{k_{i_1}}}< d^r\cdot r!,$$
		a contradiction. 
		From now on, we may assume that such monomial $\mathfrak{m}$ does not exist. Without loss of generality, we assume $r_{ij}\ge r_{i'j'}$ whenever $i<i'$. Then there exists a positive integer $m_1\le m$ such that $r_{ij}\ge 2$ for $i\le m_1$ and $r_{ij}=1$ for $i>m_1$. We set
		$$F_i:=\sum_{1\le j\le k_i} F_{ij}\;\; (i=1,2,\ldots,m_1)~~~\;{\rm and}\;\,~~~F_{m_1 +1}:=\sum_{\sum_{i>m_1, 1\le j\le k_i} d_{ij}=d} F^{(d_{11},d_{12},\ldots,d_{m k_m})}.$$
		If $m_1=m$, by definition $F_{m_1+1}=0$. Then Lemma \ref{lem:smtosm} implies that $F_{m_1+1}$ is smooth if $m_1<m$. It suffices to consider two cases: (1) $F_{ij}$ are smooth for all pairs $(i,j)$ with $i\le m_1$; (2) $F_{i_1j_1}$ is not smooth for some pair $(i_1,j_1)$ with $i_1\le m_1$.
		
		Case (1). Note that $V_i$ $(i=1,\ldots,m_1)$ are irreducible $G$-stable subspaces of dimension $r_i:={\rm dim}(V_i)=r_{i1}k_i$ and the irreducible linear groups $G_i:=p_i(G)$ preserve the smooth forms $F_i$ of degree $d$. Here $p_i: G\rightarrow \GL(V_i)$ is given by $g\mapsto g|V_i$. Let $V':=\oplus_{i>m_1} V_i$ and $p': G\rightarrow \GL(V')$ be given by $g\mapsto g|V'$. Let $G_{m_1+1}:=p'(G)$. Then $r=r_1+\cdots r_{m_1}+r_{m_1+1}$, where $r_{m_1+1}:={\rm dim}(V')$. By Proposition \ref{pp:irrbounds}, we have $\mathfrak{R}_{F_i}=\frac{|\Aut({F_i})|}{d^{r_i}\cdot r_i !}\le 5/2$ $(i=1,2,\ldots,m_1)$. Clearly, $G_i\subset \Aut({F_i})$ $(i=1,2,\ldots,m_1+1)$. Since $G_{m_1+1}$ consists of semi-permutations by $r_{ij}=1$ for $i>m_1$, by slightly adapting Theorem \ref{thm:mono}, we have $\mathcal{Q}:=\frac{|G_{m_1+1}|}{d^{r_{m_1+1}}\cdot r_{m_1+1} !}\le 1$. Then by Lemma \ref{lem:ratioprod} and $$1\le \mathfrak{R}_F=\frac{|\Aut(F)|}{d^r\cdot r!}\le \frac{|G_1|\cdot|G_2|\cdots |G_{m_1}|\cdot |G_{m_1+1}|}{d^r\cdot r!}\le \frac{(\prod_{i=1}^{m_1} \mathfrak{R}_{F_i}\cdot r_i!)\cdot (\mathcal{Q}\cdot r_{m_1+1}!)}{(r_1+\cdots+r_{m_1+1})!}, $$
		we have $m_1=m=2$, $r_1=r_2=2$ and $r=4$. Thus by Proposition \ref{pp:irrbounds} again, we infer that $H_{11}\cong H_{21}\cong A_5$ and $d=12$. Then we may write $F=F_1(x_1,x_2)+F_2(x_3,x_4)+F'(x_1,x_2,x_3,x_4)$, where $$F'\in \Oplus_{1\leq d_1\leq11}S^{d_1}(V_1^*)\otimes S^{12-d_1}(V_2^*).$$
		If $F'$ is zero, then $\Aut(F)$ is irreducible (see Theorem \ref{thm:imprim}), a contradiction. If $F'\neq 0$, then by Lemma \ref{lem:d1d2}, $H_{11}$ has a normal subgroup $\bar{H}_{11}$ and $|\Aut(F)|\le \frac{R(l(G),12)\cdot|\bar{\bar{H}}_{11}|\cdot d^r\cdot r!}{{\rm JC}(2)}=\frac{5\cdot |\bar{\bar{H}}_{11}|\cdot d^r\cdot r!}{144}$, where $\bar{\bar{H}}_{11}$ is a normal subgroup of $\bar{H}_{11}$ and $\bar{\bar{H}}_{11}$ preserves a non-zero form of degree $j$ for some $1\le j\le 11$. Since $H_{11}\cong A_5$ is a simple group, we have $\bar{\bar{H}}_{11}$ is either $H_{11}$ or trivial. From this and Lemma \ref{lem:A4S4A5}, we infer that $\bar{\bar{H}}_{11}$ is trivial and $\mathfrak{R}_F\le \frac{5}{144}<1$, a contradiction. This proves case (1).
		
		Case (2). Let $\mathcal{A}$ and $c$ be as in Lemma \ref{lem:boundtypeII}. Recall that $$\mathcal{A}=\{(i_2,j_2)|\; (i_2,j_2)\neq (i_1,j_1),\; \exists \text{ a monomial }\mathfrak{m}\in S^{d-1}(W_{i_1 j_1}^*)\otimes S^{1}(W_{i_2 j_2}^*) \text{ in }F \}$$ and $c$ is the cardinality of $\mathcal{A}$. Then $c>0$. If there is a pair $(i_2,j_2)\in\mathcal{A}$ with $r_{i_2 j_2}\ge 2$, then $\mathfrak{R}_F <1$ by Lemma \ref{lem:boundtypeII} (1) and the inequality (\ref{eq:dividingH}). So we may assume $r_{i_2,j_2}=1$ for all $(i_2,j_2)\in\mathcal{A}$. If $c\ge 2$, then by Lemmas \ref{lem:2eleinl} (3) and \ref{lem:boundtypeII} (2), we have $R(l(G),d)<3$ and $\mathfrak{R}_F\le \frac{R(l(G),d)}{d^{c-1}}<1$, which is impossible. Thus we have $c=1$ and $H_{i_1 j_1}$ admits an $F_{i_1 j_1}$-lifting preserving some non-zero form $F_1'$ of degree $d-1$. Since $l(\Aut(F))$ contains $r_{i_1 j_1}\ge 2$ and $1$,  by Lemma \ref{lem:2eleinl} (3) and Remark \ref{rem:136}, we have $r_{i_1 j_1}=2$, $3$, $4$, $6$ and $R(l(G),d)<11$, $2$, $11$, $6$ respectively. By Lemma \ref{lem:FTRandGF}, we have $$1\le \mathfrak{R}_F\le \frac{R(l(G),d)\cdot |H_{r_{i_1 j_1}}|}{{\rm JC}(r_{i_1 j_1})}\;\text{ and }\; |H_{r_{i_1 j_1}}|\ge \frac{{\rm JC}(r_{i_1 j_1})}{R(l(G),d)}.$$ From this, we have $|H_{i_1 j_1}|>60/11$, $360/2$, $25920/11$, $6531840/6$ if $r_{i_1 j_1}=2$, $3$, $4$, $6$ respectively. On the other hand, $H_{i_1 j_1}$ admitting an $F_{i_1 j_1}$-lifting implies that $H_{i_1 j_1}$ satisfies  Lemma \ref{lem:primlifting} (1)-(3), which is impossible by Table \ref{tab:listofprim} (see \cite{CCNPW} for character tables of finite simple groups).
	\end{proof}
	
	\subsection{A stronger form of the main theorem}\label{ss:7.3} Now we prove the following theorem, which in particular completes the proof of Theorem \ref{thm:Main}.
	
	\begin{theorem}\label{thm:atleastFermat}
		Fix integers $n\geq1$, $d\geq3$ with $(n,d)\neq (1,3), (2,4)$. Let $X\subset \P^{n+1}$ be a smooth hypersurface of degree $d$ with $|\Aut(X)|\ge d^{n+1}\cdot (n+2)!$. If $X$ is not isomorphic to Fermat hypersurface $X_d^n$,  then $X$ is isomorphic to one of the following smooth hypersurfaces $X_F$:
		
		\begin{footnotesize}
			\begin{table}[H]\rm
				\renewcommand\arraystretch{1.2}
				\begin{tabular}{p{0.3in}p{0.8in}p{0.6in}p{4.0in}}
					\hline
					$(n,d)$&$\Aut(X_F)$&$|\Aut(X_F)|$&$F$\\
					\hline
					$(1,4)$&$\PSL(2,7)$&$168$&$x_1^3x_2+x_2^3x_3+x_3^3x_1$ \\
					$(1,6)$&$C_3^2\rtimes \SL(2,3)$&$216$&$x_1^6 +x_2^6 + x_3^6-10(x_1^3x_2^3 + x_2^3x_3^3 + x_3^3x_1^3)$ \\
					%                \hline
					$(1,6)$&$A_6$&$360$&$10x_1^3x_2^3+9(x_1^5+x_2^5)x_3-45x_1^2x_2^2x_3^2-135x_1x_2x_3^4+27x_3^6$ \\
					$(2,6)$&$C_{6}.(S_4^2\rtimes C_2)$&$6912$&$x_1^5x_2+x_2^5x_1+x_3^5x_4+x_4^5x_3$ \\
					%                \hline
					$(2,12)$&$C_{12}.(A_5^2\rtimes C_2) $&$86400$&$x_1^{11}x_2+11x_1^6x_2^6-x_1x_2^{11}+x_3^{11}x_4+11x_3^6x_4^6-x_3x_4^{11}$ \\
					$(4,6)$&${\rm PSU}(4,3).C_2$&$6531840$&$\sum\limits_{1\leq i\leq6}x_i^6+\sum\limits_{1\leq i\neq j\leq6}15x_i^4x_j^2-\sum\limits_{1\leq i<j<k\leq6}30x_i^2x_j^2x_k^2+240\sqrt{-3}x_1x_2x_3x_4x_5x_6$ \\
					$(4,12)$&$C_{12}^2.(A_5^3\rtimes S_3) $&$186624000$&$x_1^{11}x_2+11x_1^6x_2^6-x_1x_2^{11}+x_3^{11}x_4+11x_3^6x_4^6-x_3x_4^{11}+x_5^{11}x_6+11x_5^6x_6^6-x_5x_6^{11}$ \\
					\hline
				\end{tabular}
			\end{table}
		\end{footnotesize}
		
	\end{theorem}
	
	\begin{proof}
		Let $F_0$ be a defining polynomial of $X$. Based on the exact sequence \eqref{eq:keyexact}, by \cite{MM63} and \cite{Cha78}, we have $|\Aut(X)|=|{\rm Lin}(X)|=|\Aut(F_0)|/d$. Thus, $|\Aut(F_0)|\ge d^{n+2}\cdot (n+2)!$, which implies $\Aut(F_0)$ is irreducible by Theorem \ref{thm:reducible}. Then the conclusion of the theorem follows from Theorems \ref{thm:mono}, \ref{thm:prim}, \ref{thm:imprim}.
	\end{proof}
	
	\begin{remark}\label{rem:comparebounds}
		Let $X_F\subset \P^{n+1}$ be a smooth hypersurface of degree $d\ge 3$ defined by $F$. If $n+2\ge 71$, then ${\rm J}(n+2)= (n+3)!$ by \cite[Theorem A]{Col07} (note that this bound can be achieved by $S_{n+3}$). On the other hand, by Theorem \ref{thm:diag}, the optimal upper bound  for the size of an abelian subgroup in $\Aut(F)$ is $d^{n+2}$. Thus the product ${\rm J}(n+2)\cdot d^{n+2}=d^{n+2}\cdot (n+3)!$ gives an upper bound for $|\Aut(F)|$, which gives an upper bound $d^{n+1}\cdot (n+3)!$ for $|\Aut(X_F)|$. However, this bound for $|\Aut(X_F)|$ is larger than the optimal bound $d^{n+1}\cdot (n+2)!=|\Aut(X_d^n)|$.
	\end{remark}
	
	\begin{remark}
		By Lefschetz hyperplane theorem, the Picard group of a smooth hypersurface $X$ of dimension at least $3$ is the same as for projective space. Hence $\Aut(X)={\rm Lin}(X)$. A smooth hypersurface $X$ in $\P^{n+1}$ of degree $d=1$ (resp. $2$) is isomorphic to $\P^{n}$ (resp. Fermat quadric $X_{2}^n$) with infinite automorphism group. Smooth cubic curves are elliptic curves and their automorphism groups are infinite. The automorphism group of a generic quartic surface is finite, but there are many smooth quartic surfaces with infinite automorphism groups. The automorphism group of Fermat quartic $X_4^2$ is infinite, and its finite subgroup ${\rm Lin}(X_4^2)\cong C_4^3\rtimes S_4$ determines $X_4^2$ uniquely among all K3 surfaces (\cite[Theorem 1.2 (i)]{Ogu05}).
	\end{remark}

\end{document}